\documentclass[a4paper]{article}

\usepackage{zibtitlepage}
\usepackage{algorithm,algorithmic}
\usepackage{amsthm}
\usepackage[utf8]{inputenc}

\usepackage{geometry}

\usepackage{amsmath}
\usepackage{amsfonts}
\usepackage{amssymb}
\usepackage{mathtools}
\usepackage{dsfont}
\usepackage{accents}

\usepackage{glossaries}

\renewcommand{\algorithmiccomment}[1]{\bgroup\hfill~\texttt{/* #1 */}\egroup}

\usepackage{tabularx,supertabular,booktabs,array}
\usepackage{pdflscape}
\usepackage{longtable}

\usepackage{tikz}
\usepackage{etex}
\usepackage{pgfplots}
\usetikzlibrary{shapes.gates.logic.US,shapes.geometric}
\usetikzlibrary{shapes,arrows,backgrounds,topaths}
\usetikzlibrary{arrows,decorations.markings}
\usetikzlibrary{matrix,snakes}
\usetikzlibrary{plotmarks}
\usetikzlibrary{decorations.text}
\usetikzlibrary{patterns}
\usetikzlibrary{calc}

\usepackage{pgfplots}

\usepackage{url}

\usepackage{xcolor}
\usepackage{xspace}
\usepackage{subcaption}
\usepackage{xargs}
\usepackage[colorinlistoftodos,prependcaption,textsize=tiny]{todonotes}
\usepackage{fnpct}

\usepackage{enumitem}
\newlist{abbrv}{itemize}{1}
\setlist[abbrv,1]{label=,labelwidth=1in,align=parleft,itemsep=0.1\baselineskip,leftmargin=!}

\usepackage{proofread}
 
\newcommand{\name}[1]{\mbox{#1}\xspace}
\newcommand{\scip}{\name{SCIP}}
\newcommand{\antigone}{\name{ANTIGONE}}
\newcommand{\baron}{\name{BARON}}
\newcommand{\couenne}{\name{Couenne}}

\newcommand{\cplex}{\name{CPLEX}}
\newcommand{\cppad}{\name{CppAD}}
\newcommand{\ipopt}{\name{Ipopt}}
\newcommand{\mumps}{\name{Mumps}}
\newcommand{\globallib}{\name{GLOBALLib}}
\newcommand{\minlplib}{\name{MINLPLib}}
\newcommand{\minlplibtwo}{\name{MINLPLib2}}

\newcommand{\MINLP}{\name{MINLP}}
\newcommand{\MINLPs}{\name{MINLPs}}
\newcommand{\MIQCP}{\name{MIQCP}}
\newcommand{\MIQCPs}{\name{MIQCPs}}

\newcommand{\NLP}{\name{NLP}}

\newcommand{\LP}{\name{LP}}
\newcommand{\LPs}{\name{LPs}}
\newcommand{\OBBT}{\name{OBBT}}
\newcommand{\FBBT}{\name{FBBT}}
\newcommand{\RLT}{\name{RLT}}
\newcommand{\SDP}{\name{SDP}}
\newcommand{\SDPs}{\name{SDPs}}

\newcommand{\T}{\mathsf{T}}
\newcommand{\mvec}[2]{\begin{pmatrix}#1\\#2\end{pmatrix}}

\newcommand{\R}{\mathbb{R}}
\newcommand{\Z}{\mathbb{Z}}

\DeclareMathOperator{\sign}{sign}
\newcommand{\volume}[1]{\text{Vol}(#1)}

\newcommand{\dash}{---}

\newcommand{\tabledefline}[2]{\multicolumn{1}{l}{\rlap{#1\ \dash\ #2}}\\}

\definecolor{plum}{rgb}{0.8, 0.6, 0.8}
\definecolor{applegreen}{rgb}{0.55, 0.71, 0.0}
\definecolor{apricot}{rgb}{0.98, 0.81, 0.69}
\definecolor{amber}{rgb}{1.0, 0.49, 0.0}
\definecolor{americanrose}{rgb}{1.0, 0.01, 0.24}

\newcommandx{\change}[2][1=]{\todo[linecolor=americanrose,backgroundcolor=americanrose!25,bordercolor=americanrose,#1]{#2}\,}
\newcommandx{\improvement}[2][1=]{\todo[linecolor=amber,backgroundcolor=amber!25,bordercolor=amber,#1]{#2}\,}
\newcommandx{\unsure}[2][1=]{\todo[linecolor=apricot,backgroundcolor=apricot!25,bordercolor=apricot,#1]{#2}\,}
\newcommandx{\info}[2][1=]{\todo[linecolor=applegreen,backgroundcolor=applegreen!25,bordercolor=applegreen,#1]{#2}\,}
\newcommandx{\missing}[2][1=]{\todo[linecolor=blue,backgroundcolor=blue!25,bordercolor=blue,#1]{#2}\,}

\newcommand{\fa}{\text{ for all }}
\newcommand{\psd}{\succeq}

\DeclareMathOperator{\convexhull}{conv}
\DeclareMathOperator*{\argmax}{argmax}
\DeclareMathOperator*{\argmin}{argmin}
\newcommand{\NP}{$\mathcal{NP}$}

\newcommand{\consindex}{\mathcal{M}}
\newcommand{\varindex}{\mathcal{N}}
\newcommand{\intvarindex}{\mathcal{I}}
\newcommand{\nvars}{n}
\newcommand{\nconss}{m}
\newcommand{\lb}{\ell}
\newcommand{\ub}{u}
\newcommand{\lbi}{\lb_i}
\newcommand{\ubi}{\ub_i}
\newcommand{\lbj}{\lb_j}
\newcommand{\ubj}{\ub_j}
\newcommand{\ibox}{[\lb_i,\ub_i]}
\newcommand{\jbox}{[\lb_j,\ub_j]}
\newcommand{\ijbox}{[\lb_i,\ub_i] \times [\lb_j,\ub_j]}

\newcommand{\feasset}{\mathcal{F}}

\newcommand{\relaxset}{\mathcal{X}}
\newcommand{\projrelaxset}{\relaxset_{ij}}
\newcommand{\projfeasset}{\feasset_{ij}}
\newcommand{\incumbent}{\mathcal{U}}
\renewcommand{\xi}{x_i}
\newcommand{\xj}{x_j}
\newcommand{\xibar}{\bar x_i}
\newcommand{\xjbar}{\bar x_j}
\newcommand{\Xij}{X_{ij}}
\newcommand{\Pij}{P_{ij}}
\newcommand{\varbounds}[1]{[\lb_{#1},\ub_{#1}]}
\newcommand{\diaglp}[2]{\LP(#1,#2)}
\newcommand{\centerp}{C}
\newcommand{\gapfct}{GC}

\newcommand{\ninstances}{1682}
\newcommand{\expAffected}{\texttt{AFFECTED}}
\newcommand{\expRootgap}{\texttt{ROOTGAP}}
\newcommand{\expTree}{\texttt{TREE}}

\newcommand{\treeScip}{\texttt{SCIP}}
\newcommand{\treeScipSepa}{\texttt{SCIP+s}}
\newcommand{\treeScipSepaProp}{\texttt{SCIP+s+p}}

\newcommand{\revision}[1]{#1}
\newcommand{\revisionstart}{}
\newcommand{\revisionend}{} \newcommand{\affectedNNoBilins}{464}

\newcommand{\affectedNAffected}{564}

\newcommand{\affectedNAffectedZero}{97}

\newcommand{\affectedNAffectedFull}{82}

\newcommand{\affectedArithAll}{0.19}

\newcommand{\affectedArithAllPercent}{18.9}

\newcommand{\affectedArithAffected}{0.40}

\newcommand{\affectedArithAffectedPercent}{40.3}

\newcommand{\affectedGeomeanTimeAll}{2.6}
\newcommand{\affectedGeomeanTimeAffected}{3.6}
\newcommand{\affectedGeomeanTimeNotffected}{1.0}
\newcommand{\affectedGeomeanItersAll}{4454.6}
\newcommand{\affectedGeomeanItersAffected}{9374.3}
\newcommand{\affectedGeomeanItersNotffected}{875.3}

\newcommand{\affectedFilteredInstances}{797}
\newcommand{\affectedFilteredRateAllPercentage}{48.1}
\newcommand{\affectedFilteredRateFilteredPercentage}{51.0}
 \newcommand{\rootNinstances}{$547$}

\newcommand{\rootGapClosedAllPercent}{$7.5$}

\newcommand{\rootGapClosedAffectedPercent}{$20.8$}
\newcommand{\rootGapClosedAffectedNInstances}{$178$}

\newcommand{\rootGapClosedBetterPercent}{$24.9$}
\newcommand{\rootGapClosedBetterNInstances}{$165$}

\newcommand{\rootGapClosedWorsePercent}{$-31.1$}
\newcommand{\rootGapClosedWorseNInstances}{$13$}
 
\newcommand{\treeNInstancesAll}{564}

\newcommand{\treeNSolvedAllDefault}{249}

\newcommand{\treeNSolvedAllNoprop}{247}
\newcommand{\treeTimeAllNoprop}{+0\%}

\newcommand{\treeNSolvedAllOff}{244}
\newcommand{\treeTimeAllOff}{+3\%}

\newcommand{\treeNInstancesOne}{166}

\newcommand{\treeNSolvedOneDefault}{159}

\newcommand{\treeNSolvedOneNoprop}{158}
\newcommand{\treeTimeOneNoprop}{+1\%}

\newcommand{\treeNSolvedOneOff}{155}
\newcommand{\treeTimeOneOff}{+1\%}

\newcommand{\treeNInstancesTen}{109}

\newcommand{\treeNSolvedTenDefault}{102}

\newcommand{\treeNSolvedTenNoprop}{101}
\newcommand{\treeTimeTenNoprop}{+2\%}

\newcommand{\treeNSolvedTenOff}{98}
\newcommand{\treeTimeTenOff}{+18\%}

\newcommand{\treeNInstancesHundred}{44}

\newcommand{\treeNSolvedHundredDefault}{38}

\newcommand{\treeNSolvedHundredNoprop}{40}
\newcommand{\treeTimeHundredNoprop}{-3\%}

\newcommand{\treeNSolvedHundredOff}{33}
\newcommand{\treeTimeHundredOff}{+36\%}
 
\newtheorem{theorem}{Theorem}
\newtheorem{remark}{Remark}
\newtheorem{lemma}{Lemma}
\newtheorem{definition}{Definition}

\newtheorem{claim}{Claim}

\begin{document}

\ZTPAuthor{
  \ZTPHasOrcid{Benjamin Müller}{0000-0002-4463-2873},
  \ZTPHasOrcid{Felipe Serrano}{0000-0002-7892-3951},
  \ZTPHasOrcid{Ambros Gleixner}{0000-0003-0391-5903}}
\ZTPTitle{Using two-dimensional Projections for Stronger Separation and Propagation of Bilinear Terms}
\ZTPNumber{19-15}
\ZTPMonth{March}
\ZTPYear{2019}

\title{\bf Using two-dimensional Projections for Stronger Separation and Propagation of Bilinear Terms}
\author{\ZTPHasOrcid{Benjamin Müller}{0000-0002-4463-2873}\and\
  \ZTPHasOrcid{Felipe Serrano}{0000-0002-7892-3951}\and\
  \ZTPHasOrcid{Ambros Gleixner}{0000-0003-0391-5903}}
\hypersetup{pdftitle={Using two-dimensional Projections for Stronger Separation and Propagation of Bilinear Terms},
  pdfauthor={Benjamin Müller, Felipe Serrano, Ambros Gleixner}}

\zibtitlepage
\maketitle

\begin{abstract}
  One of the most fundamental ingredients in mixed-integer nonlinear programming solvers is the well-known
McCormick relaxation for a product of two variables $x$ and $y$ over a box-constrained domain. 
The starting point of this paper is
the fact that the convex hull of the graph of $xy$ can be much tighter when computed over a strict, non-rectangular subset of the box.  In order to exploit this in practice, we propose to compute valid
linear inequalities for the projection of the feasible region onto the $x$-$y$-space by solving a sequence of linear
programs akin to optimization-based bound tightening.  These valid inequalities allow us to employ results from the literature 
to strengthen the classical McCormick relaxation.  As a consequence, we obtain a
stronger convexification procedure that exploits problem structure and can benefit from supplementary information obtained
during the branch-and bound algorithm such as an objective cutoff.
We complement this by a new bound tightening procedure that efficiently computes the best possible bounds for $x$, $y$,
and $xy$ over the available projections.
Our computational evaluation using the academic solver \scip exhibit that the proposed methods are applicable to a large portion of the public test library \minlplib and help to improve performance significantly.
 \end{abstract}

\section{Introduction}

This paper is concerned with solving nonconvex mixed-integer quadratically constrained programs (\MIQCP{}s) of the form
\begin{equation}
  \begin{aligned}
    \min \quad& c^\T x \\
    \text{s.t.} \quad& x^\T Q_k x + q_k^\T x \le b_k && \fa k \in \consindex, \\
    &\lb_{i} \le \xi \leq \ub_{i} && \fa i \in \varindex, \\
    &\xi \in \Z && \fa i \in \intvarindex,
  \end{aligned}
  \label{eq:miqcp}
\end{equation}
where~$\varindex := \{1,\ldots,\nvars\}$ is the index set of variables,~$\consindex := \{1,\ldots,\nconss\}$ the index set
of constraints, ~$c\in \R^{\nvars}$ is the objective function vector,~$\lb \in \R^n$ and~$\ub \in \R^n$ are the vectors
of lower and upper bounds of the variables, $\intvarindex \subseteq \varindex$ is the index set of integer
variables, and $Q_k \in \R^{n \times n}$ is a symmetric matrix for each $k \in \consindex$.
Many real-world applications are inherently nonlinear and need to be tackled as~\MIQCP{}s or general mixed-integer nonlinear programs (\MINLP{}s) that include quadratic constraint functions.
For a selection
see, e.g.,~\cite{GrossmanSahinidis2002}.
In this article, we develop new convexification and bound tightening techniques that are directly relevant to achieve
improvements within the algorithmic framework of spatial branch-and-bound, which forms the basis of many modern solvers
in global optimization, e.g., \antigone~\cite{ANTIGONE}, \baron~\cite{Sahinidis2017}, \couenne~\cite{Couenne}, and
\scip~\cite{SCIP}.

For clarity of presentation we assume that the \MIQCP is equivalently reformulated as
\begin{equation}
  \begin{aligned}
    \min \quad & c^\T x \\
    \text{s.t.} \quad & \langle X,Q_k\rangle + q_k^\T x \le b_k && \fa k \in \consindex, \\
    &\lb_{i} \le \xi \leq \ub_{i} && \fa i \in \varindex, \\
    &\xi \in \Z && \fa i \in \intvarindex, \\
    & X = xx^\T.
  \end{aligned}
  \label{eq:miqcpref}
\end{equation}
This reformulation is obtained by linearizing the original
quadratic constraints via auxiliary variables and new constraints of the form $\Xij = \xi\xj$ for $i,j \in \varindex$. Usually, these
constraints are only added for those $i,j \in \varindex$ for which $\xi \xj$ appears in at least one of the quadratic
constraints of~\eqref{eq:miqcp}, i.e., if $(Q_k)_{ij} \neq
0$ for some $k \in \consindex$. Formulation~\eqref{eq:miqcpref} is of major importance when using convex relaxations for solving \MIQCPs to global
optimality and allows us to focus on tight relaxations for the elementary nonconvex constraints of the form $\Xij = \xi \xj$ with $i \neq j$.
The techniques presented in this paper extend fully to such bilinear constraints present in general reformulations that are applied when solving factorable \MINLPs to global optimality~\cite{Ryoo1995,Smith1997,Belotti2009}.
For example, when a
nonlinear constraint of the form $f(x) \, g(x) \le d$ is reformulated as
\begin{equation}
  w_1 = f(x), \; w_2 = g(x), \; w_1 w_2 \le d,
\end{equation}
with auxiliary variables $w_1,w_2 \in \R$, our results can be directly applied
to improve the convexification and propagation of the product $w_1 w_2$.

Our initial motivation is as follows.
Classically, a linear relaxation for the nonconvex constraint $\Xij = \xi \xj$, $i \not= j$, is constructed by adding the four inequalities
\begin{equation}\label{eq:mccormick}
\begin{aligned}
  \Xij &\ge \ub_j \xi + \ub_i \xj - \ub_i \ub_j, \\
  \Xij &\ge \lb_j \xi + \lb_i \xj - \lb_i \lb_j, \\
  \Xij &\le \ub_j \xi + \lb_i \xj - \lb_i \ub_j, \\
  \Xij &\le \lb_j \xi + \ub_i \xj - \ub_i \lb_j,
\end{aligned}
\end{equation}
often called McCormick inequalities~\cite{McCormick1976}.
These inequalities are best possible on
the domain $\ijbox$ in the sense that they describe the convex and concave envelope of $\xi\xj$~\cite{Khayyal1983}.
However, they do not take into account the presence of other linear and
nonlinear inequalities of~\eqref{eq:miqcpref}.

Suppose that for all feasible points $(x,X)$
of~\eqref{eq:miqcpref} the points $(\xi,\xj)$ are contained in a polytope $P$ that is a strict subset of $\ijbox$. As can be seen in Figure~\ref{fig:introexamples}a,
the convex hull of the graph of $\xi\xj$ over $P$ is not given by~\eqref{eq:mccormick} and is not
polyhedral. Tangent inequalities for the convex and concave envelope of $\xi\xj$ over $P$ lead to a stronger linear
relaxation of $\Xij = \xi \xj$ than~\eqref{eq:mccormick}.
\begin{figure}[t]
  \centering
  \begin{minipage}[t]{0.475\textwidth}
    \centering
    \includegraphics[width=0.8\textwidth]{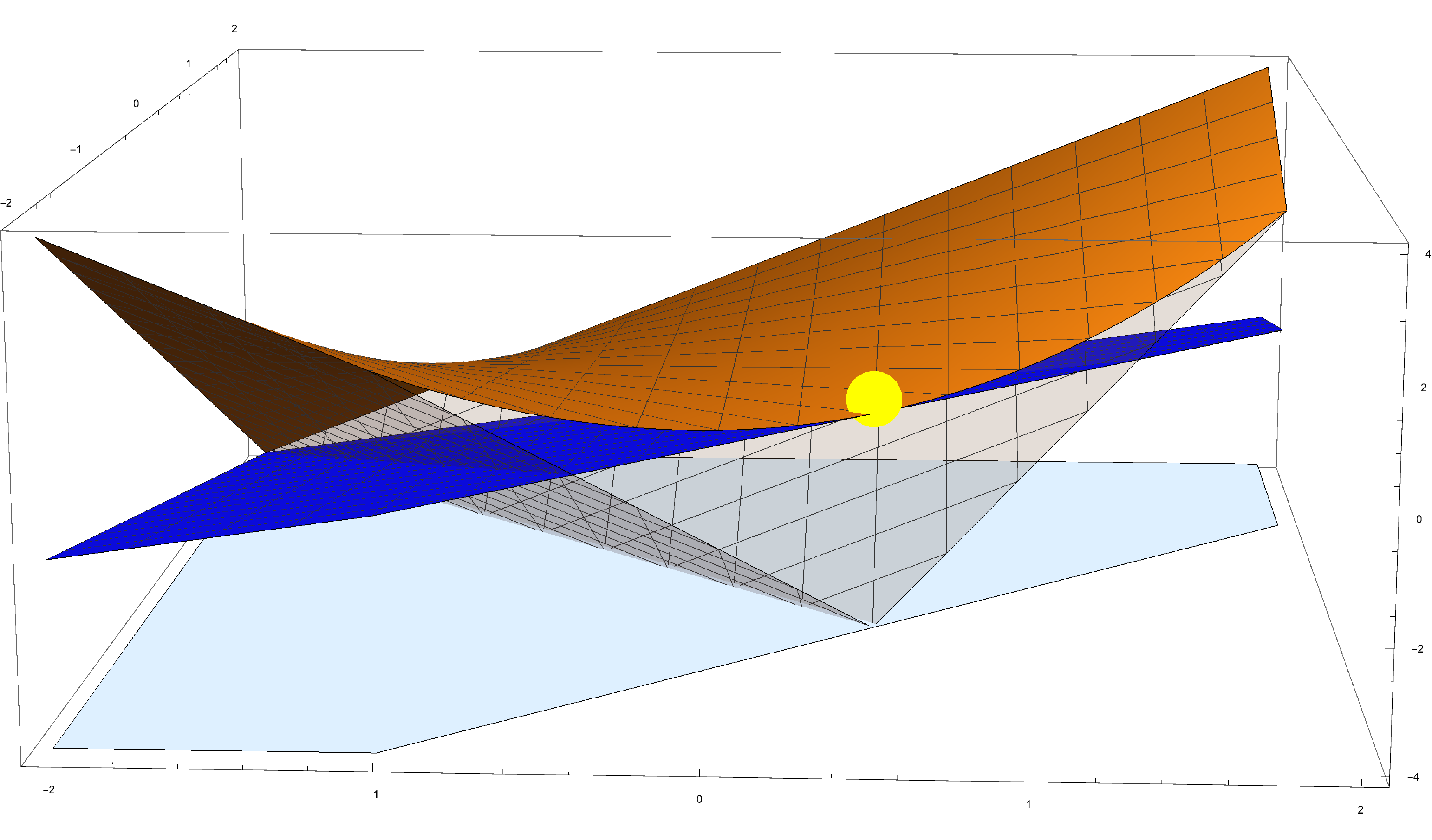}
  \subcaption{Improving separation: Two functions that are valid underestimators for $\xi\xj$ (orange) on a polyhedral
    domain (cyan). The figure shows that a linearization of the convex envelope (blue) at a given point (yellow) is
    locally tighter than the McCormick relaxation (gray).}
  \end{minipage}
  \hfill
  \begin{minipage}[t]{0.475\textwidth}
    \centering
    \raisebox{3.5ex}{\includegraphics[width=0.9\textwidth]{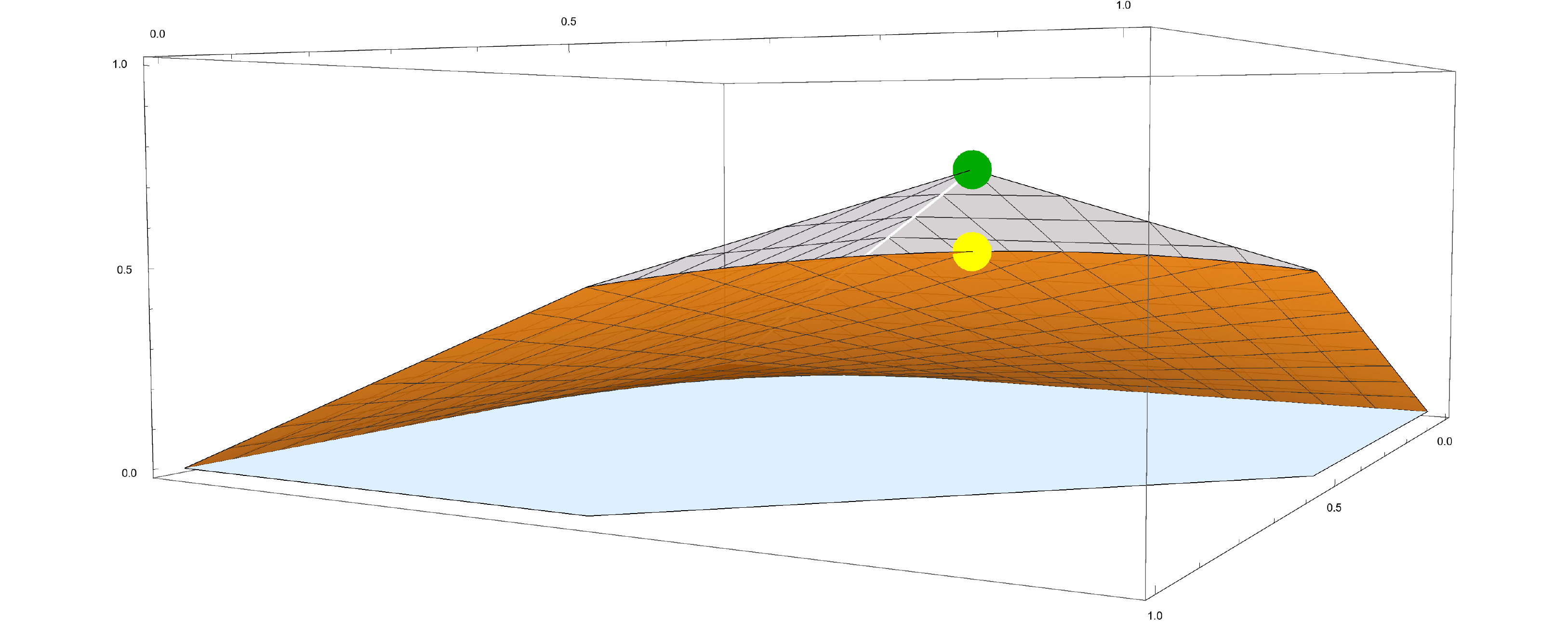}}
    \subcaption{Improving propagation: The colored plot shows $\xi\xj$ over the polytope
    $P = [0,1]^2 \cap \{(\xi,\xj) \mid \xi + \xj \le \frac{3}{2} \}$ (cyan). The yellow point corresponds to the best
    possible bound of $\xi\xj$ on $P$, which is better than the best possible bound implied by the McCormick
    relaxation~\eqref{eq:mccormick} on $P$ (green point).}
  \end{minipage}
  \caption{Both separation and propagation can be improved by exploiting the
    presence of a non-rectangular, polyhedral domain.}
  \label{fig:introexamples}
\end{figure}

In addition to tighter underestimators, knowledge about~$P$ can be exploited to construct tighter variable bounds. For
example, consider the polytope
\begin{equation}
  P = \{ (\xi,\xj) \in [0,1]^2 \mid \xi + \xj \le 3/2 \}.
\end{equation}
The best possible upper bound for $\Xij = \xi \xj$ over $P$ is given as
\begin{equation}\label{eq:propexample}
  \max \{ \Xij \mid (\xi,\xj) \in P, \, \Xij = \xi \xj \} = \frac{9}{16}.
\end{equation}
This improves upon the upper bound implied by the McCormick
relaxation over $P$,
\begin{equation}
  \max \{ \Xij \mid (\xi,\xj) \in P, \eqref{eq:mccormick} \} = \frac{3}{4}.
\end{equation}
An illustration is given in Figure~\ref{fig:introexamples}b.

These two examples show that a two-dimensional polytope $P \subsetneq \ijbox$ for $(\xi,\xj)$ can be exploited in order to improve the convexification and propagation of $\Xij = \xi \xj$. In order to leverage this potential in practice, one needs to
determine how to efficiently compute
\begin{enumerate}
\item a suitable polytope $P$,
\item tangent inequalities for the convex and concave envelope of $\xi\xj$ over $P$, and
\item tighter variable bounds for $\xi$, $\xj$, and $\Xij$ over $P$.
\end{enumerate}
For the second step, an algorithm to compute tangent inequalities for the envelopes of
$\xi\xj$ over $P$ is presented in the recent paper by Locatelli~\cite{Locatelli2018}. One requirement of this algorithm is that $P$ needs to be explicitly given, as output of step one.

Ideally, the original formulation~\eqref{eq:miqcpref} already contains inequalities that only depend on the two variables of a
bilinear term. A good example are symmetry-breaking inequalities in circle packing problems.
For example, the instance \texttt{pointpack08} from the
\minlplib~\cite{MINLPLIB} test library contains constraints of the form
\begin{equation}\label{eq:example:packing}
  \begin{aligned}
    (x_1 - x_2)^2 + (y_1 - y_2)^2 &\ge 1, \\
    x_1 - x_2 &\le 0, \\
    (x_1, x_2, y_1, y_2) &\in [0,1]^4.\\
  \end{aligned}
\end{equation}
Here $(x_1,y_1)$ and $(x_2,y_2)$ are the centers of two circles. The quadratic constraint ensures a minimum
distance between these centers and the linear constraint orders them along the $x$-axis. In this case the inequality $x_1 \le x_2$
can be directly used for convexifying $x_1 x_2$ with Locatelli's algorithm.

However, for many instances  inequalities only depending on variables of a single bilinear term
may not appear in the initial formulation of the \MIQCP. Nevertheless, it might be possible that a substructure
of~\eqref{eq:miqcpref} implies such inequalities.
For example, consider the instance \texttt{crudeoil\_lee1\_05} from \minlplib. Aggregating the linear constraints
\begin{equation}
\begin{aligned} x_{260} + x_{292} + x_{324} + x_{356} + x_{451} & \le 50, \\
  -x_{394} + x_{525} + x_{526} + x_{527} & = 0,\\
  x_{260} + x_{292} + x_{324} + x_{451} + x_{527} & = 50,\\
  x_{525} & \ge 0, \\
  x_{526} & \ge 0,
\end{aligned}
\end{equation}
with the multiplier vector $(-\frac{1}{3},-\frac{1}{3},\frac{1}{3}, \frac{1}{3}, \frac{1}{3})$ shows that $x_{356} \le x_{394}$ is valid and thus it can
be used for strengthening the relaxation of $X_{356,394} = x_{356} \, x_{394}$.

In this spirit, the first contribution of this paper is a fully general scheme for computing projected relaxations $P$ in step one above.
It solves a sequence of linear
programs (\LPs) to compute a polyhedral relaxation of the projection of the feasible region onto the space of two
variables that appear bilinearly. The computed two-dimensional relaxation is described by at most eight inequalities.
Second, we introduce a bound tightening procedure for forward and backward propagation that solves a reduced nonconvex optimization problem.
\revision{This results in the best possible bounds for a bilinear term and its variables using the linear inequalities of the two-dimensional projection.}
Due to the construction of the projections, these optimization problems can be solved by inspecting only a constant number of points.
Last, we propose an effective way of incorporating these techniques into an \LP-based spatial branch-and-bound
solver and provide a detailed computational analysis of their impact.

The remainder of the paper is organized as follows. Section~\ref{section:literature} discusses relevant literature and
provides an overview of convex relaxations for~\eqref{eq:miqcpref}.
In Section~\ref{section:computeprojection}, we present a procedure for computing valid inequalities for the projections
of the feasible region onto the space of two variables.
Section~\ref{section:propagation} is dedicated to a bound tightening algorithm that exploits the computed projections.
Section~\ref{section:experiments} provides a thorough computational study using the \MINLP solver \scip on publicly available benchmark instances based on three experiments. First, we measure the basic potential of the methods by analyzing how many instances of \minlplib actually admit non-trivial
two-dimensional projections. Second, we study the dual bound improvement in the root node of the branch-and-bound
tree. Third, we evaluate the overall performance impact of the new methods on the full spatial branch-and-bound search.
Section~\ref{section:conclusion} gives concluding remarks.

\section{Background}
\label{section:literature}

In this section, we give a brief overview of the relevant literature. First, we review important convex relaxations for \MIQCPs and existing convexification methods for special nonconvex functions over non-rectangular domains. Second, we discuss basic bound tightening algorithms and their relation to
convexification methods. Finally, we give a short summary of Locatelli's algorithm and its complexity.

\paragraph{Convex relaxations for \MIQCPs}

Two important convex relaxations for \MIQCPs that have been exhaustively studied in the literature are semidefinite
programming (\SDP)~\cite{Vandenberghe1996} and the reformulation-linearization technique (\RLT)~\cite{Sherali1999}.
Both relaxations utilize the $\Xij$ variables of~\eqref{eq:miqcpref} in order to linearize $\xi\xj$.
For an \SDP relaxation the nonconvex constraint $X = xx^\T$ is relaxed to the convex constraint $X \psd xx^\T$, which is
equivalent to
\begin{equation*}
  \begin{bmatrix} 1 & x^\T \\ x & X \end{bmatrix} \psd 0,
\end{equation*}
via the Schur complement~\cite{Boyd2004}. Even though the resulting \SDP relaxation is
efficiently solvable in theory, optimizing \SDPs in practice is a numerically challenging task. We refer
to~\cite{Poljak1995,Fujie1997,Helmberg2000,Lemarechal2001,Luo2010,Bao2011} for applications which utilize \SDP
relaxations to solve quadratic optimization problems.

While the construction of an \SDP relaxation is independent of any linear or linearized constraints,
an \RLT-based relaxation uses them directly. After introducing auxiliary variables
$X \in \R^{\nvars \times \nvars}$ and the nonconvex constraints $X = xx^\T$, the idea is to linearize the product of all
selections of two linear inequalities with the help of $X$. For example, consider the inequalities $x_i \ge 0$ and
$\alpha^\T x - \alpha_0 \ge 0$. Multiplying the second inequality by $x_i$ gives
\begin{equation*}
  \sum_{j=1}^\nvars \alpha_j \xi \xj - \alpha_0 \xi \ge 0,
\end{equation*}
which is linearized with $X$ to
\begin{equation*}
  \sum_{j=1}^\nvars \alpha_j \Xij - \alpha_0 \xi \ge 0.
\end{equation*}
These \emph{\RLT inequalities} can significantly improve a relaxation of~\eqref{eq:miqcpref},
see~\cite{SheraliFraticelli2002,Meyer2006,Anstreicher2009}. Note that the McCormick relaxation~\eqref{eq:mccormick} is
a special form of \RLT that uses variable bound constraints only.

To obtain a convex relaxation for~\eqref{eq:miqcp}, it is not mandatory to reformulate the \MIQCP
into~\eqref{eq:miqcpref}. Following the ideas of McCormick~\cite{McCormick1976}, Vigerske~\cite{Vigerske2013} uses
linear underestimators $\tilde f_k : [\lb,\ub] \rightarrow \R$ for each nonlinear function
$f_k: [\lb,\ub] \rightarrow \R$ of a constraint $\sum_k f_k(x) \le 0$ and obtains the valid cut $\sum_k \tilde f_k(x) \le 0$ by summing the
underestimators.
The advantage of this approach is that it does not require the additional variables $X$ but Anstreicher~\cite{Anstreicher2012} shows
 that even when replacing each quadratic function with its convex envelope, this is in general
weaker than exploiting the extended formulation.

\paragraph{Convexification of bilinear terms}

Although \RLT-based relaxations utilize the \LP relaxation, they do not necessarily describe the convex hull of the constraint $\Xij = \xi\xj$ over this relaxation. For example, consider the set
\begin{equation} \label{eq:rlt:example}
  \{(\xi,\xj,\Xij) \in [0,1]^3 \mid \Xij = \xi\xj, \, \xi \le \xj \}.
\end{equation}
The \RLT relaxation of~\eqref{eq:rlt:example} is equal to
\begin{equation*}
  \{(\xi,\xj,\Xij) \in [0,1]^3 \mid \eqref{eq:mccormick},\, \xi^2 \le \Xij \},
\end{equation*}
when keeping the convex constraint $\xi^2 \le \Xij$. However, the convex hull of~\eqref{eq:rlt:example}
is given by
\begin{equation*}
  \{(\xi,\xj,\Xij) \in [0,1]^3 \mid \eqref{eq:mccormick},\, \xi^2 \le (1 + \xi - \xj) \Xij \},
\end{equation*}
which is strictly tighter. This shows that \RLT does not fully exploit the presence of linear inequalities.

In the literature, different cases for convexifying a bilinear term over special sets have been studied:
Linderoth~\cite{Linderoth2005} proposed a branch-and-bound algorithm for solving nonconvex quadratically-constrained
quadratic programs. Variables of a bilinear term are partitioned into two-dimensional triangles and rectangles. He
characterized the convex and concave envelope of $\xi\xj$ over a triangular domain and used it to improve
upon~\eqref{eq:mccormick}.
Based on perspective functions, Hijazi~\cite{Hijazi2019} derived a closed formula for the convex and concave envelope
on a polytope of the form $P := \{ (\xi,\xj) \in \ijbox \mid \xi \le \xj \}$.
As mentioned above, an algorithm for computing tangent inequalities for the convex and concave envelope of $\xi\xj$ on a
general two-dimensional polytope $P$ has been presented by Locatelli~\cite{Locatelli2018}.
Instead of using information on $\xi$ and $\xj$, Miller et al.~\cite{Miller2011} showed a lifting procedure for cutting
planes for $\Xij = \xi\xj$ that exploits bounds on $\Xij$ that are not implied by $\xi\xj$.

\paragraph{Bound tightening methods}

As it is shown in~\eqref{eq:mccormick}, there is an interdependency between the variable bounds and the strength of the
(convex) relaxation. Tighter variable bounds result in tighter relaxations for nonconvex constraints and vice versa. The
two most practically relevant methods to tighten variable bounds are \emph{feasibility-based bound tightening} (\FBBT) and
\emph{optimization-based bound tightening}~\cite{QuesadaGrossmann1995} (\OBBT). \FBBT is based on interval arithmetic,
see, e.g.,~\cite{BelottiCafieriLeeLiberti2012TR}, and computes activities of nonlinear expressions over the domain of
the variables (forward propagation) and conversely propagating the bounds on the constraint activities back to the
bounds of the variables (reverse propagation). Implementations usually rely on the representation of nonlinear term as
nodes of a directed acyclic expression graph, see~\cite{Belotti2009} or~\cite{Vigerske2013} for
details. \revision{\OBBT computes tighter lower and upper bounds for a variable $\xi$ by minimizing and maximizing $\xi$
over a linear relaxation of~\eqref{eq:miqcpref}. These two linear programs are called \emph{\OBBT \LPs}.}
Computing the best possible bounds for all variables \revision{over a fixed linear relaxation} requires solving $2\nvars$
many \OBBT \LPs and thus \OBBT is often too expensive to be applied in every node of a branch-and-bound tree. Gleixner et
al.~\cite{Gleixner2017} show how dual solutions of \OBBT \LPs can be used during the tree search as a fast
approximation of \OBBT.

\paragraph{Locatelli's algorithm}

Let $P \subset \R^2$ be a polytope and let $(\xi^*,\xj^*) \in P$. Locatelli showed that computing a tangent inequality
of the convex and concave envelope of $\xi\xj$ at $(\xi^*,\xj^*)$ reduces to selecting at most three points in the
boundary of $P$ such that $(\xi^*,\xj^*)$ is contained in the convex hull of these points. Figure~\ref{fig:locatelli}
shows all possible cases that can occur. The resulting inequality is determined by either

\begin{enumerate}
\item three vertices of $P$,
\item a vertex and a point $p$ on a facet of $P$ such that the inequality is tangent at $p$, or
\item two points $p$ and $q$ on different facets of $P$ such that the inequality is tangent at $p$ and $q$.
\end{enumerate}

Locatelli derived closed formulas for computing the inequalities in each of the three cases.
When $P = \ijbox$, they collapse to the first case and yield the McCormick
inequalities~\eqref{eq:mccormick}.
The third case only occurs if $P$ is described by at least two non-axis parallel
facets that have both a positive or both a negative slope.

To find a valid inequality that is also tangent to the convex (concave) envelope, one needs to iterate through all
possible selections of the points as discussed above, and select the inequality that has the smallest (largest) value
at $(\xi^*,\xj^*)$. The computational cost for iterating through all possible choices and computing the tangent
inequality is
\begin{equation*}
  O\left(\mvec{|V|}{3} + |V| \cdot |F| + \mvec{|F|}{2}\right),
\end{equation*}
where $|V|$ is the number of vertices and $|F|$ be the number of facets of $P$ that are not axis-parallel.

\begin{figure}[t]
  \centering
  \begin{minipage}{0.28\textwidth}
    \centering
    \includegraphics[width=\textwidth]{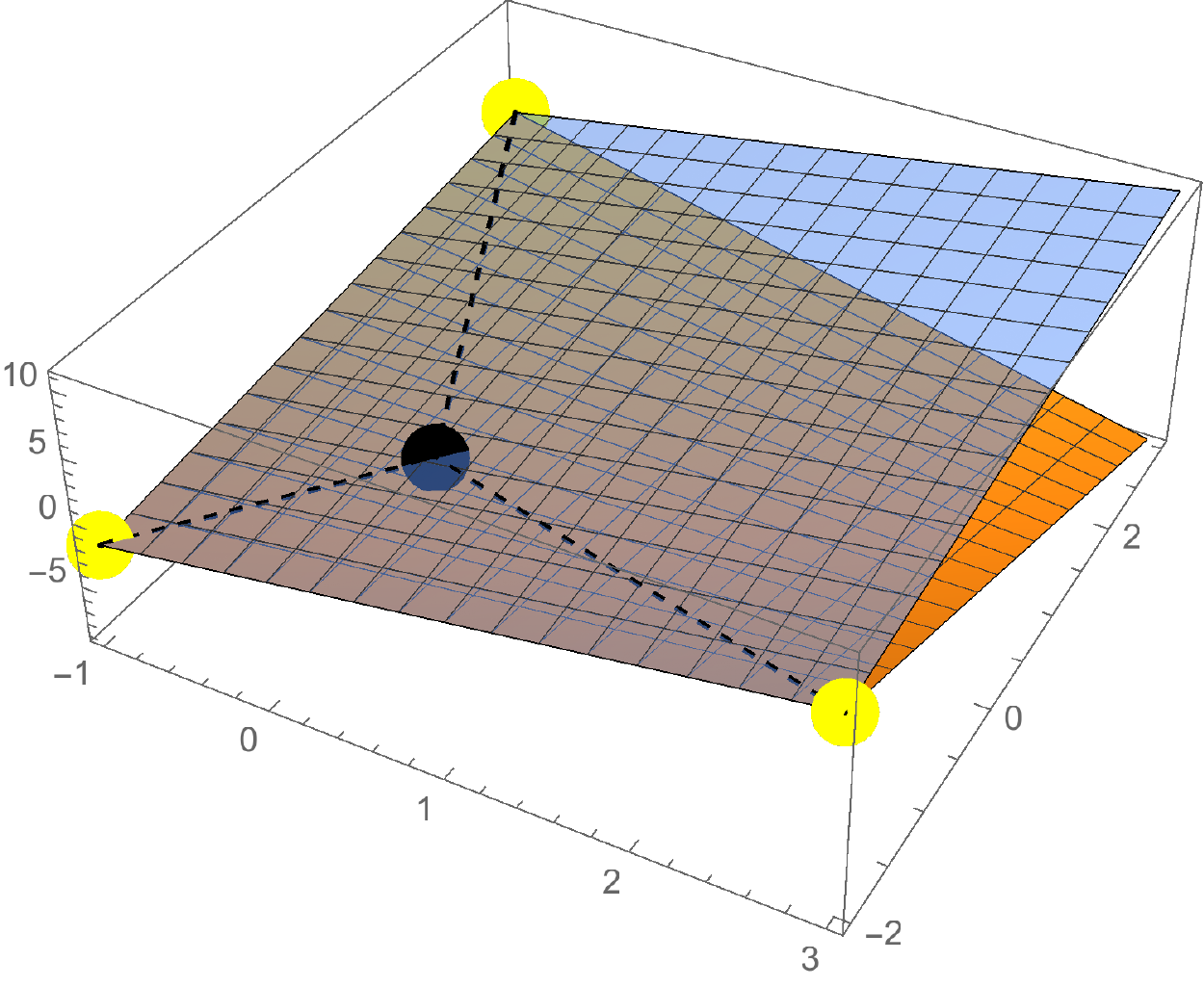}
  \end{minipage}
  \hspace{2ex}
  \begin{minipage}{0.3\textwidth}
    \centering
    \includegraphics[width=\textwidth]{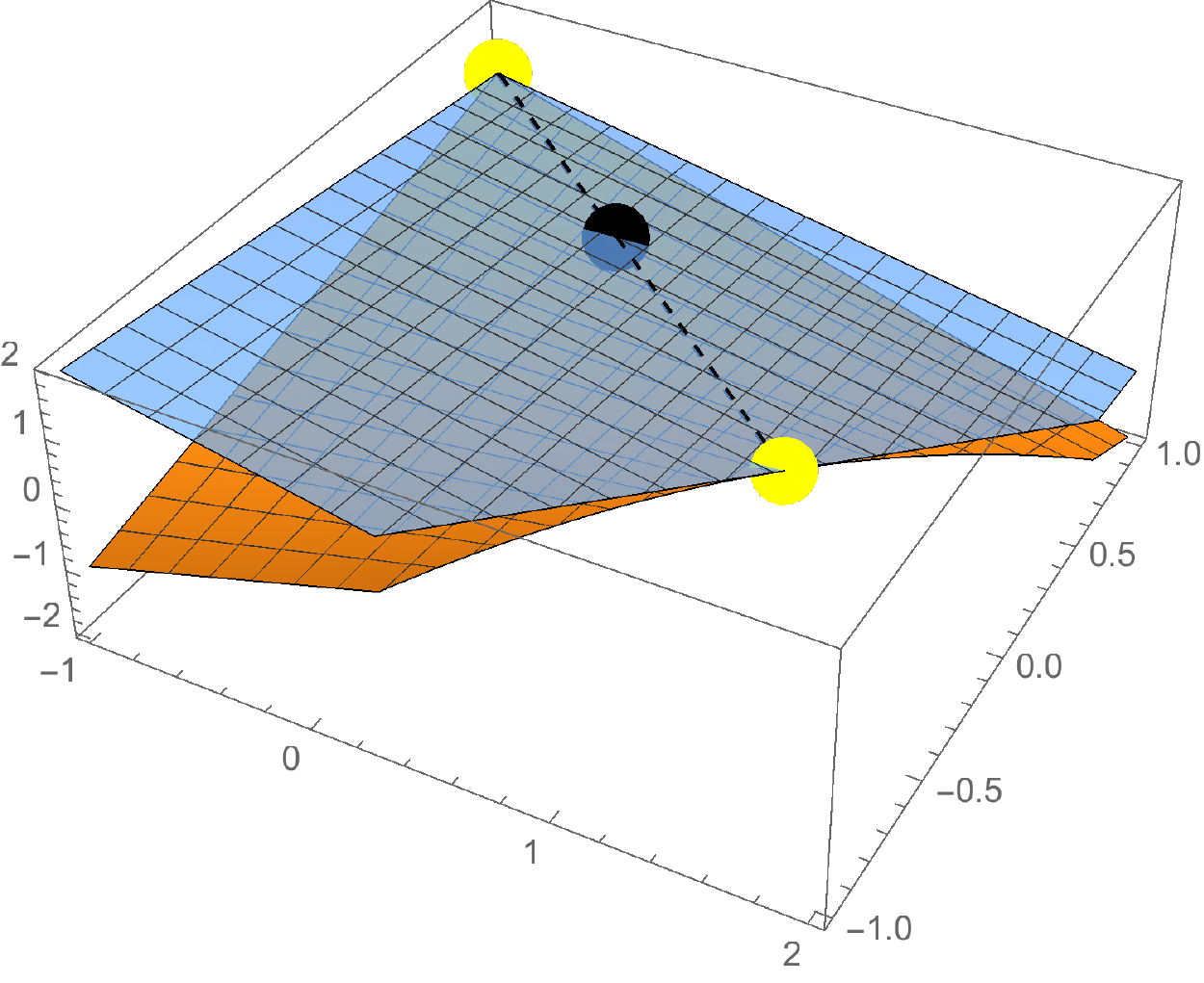}
  \end{minipage}
  \hspace{2ex}
  \begin{minipage}{0.3\textwidth}
    \centering
    \includegraphics[width=\textwidth]{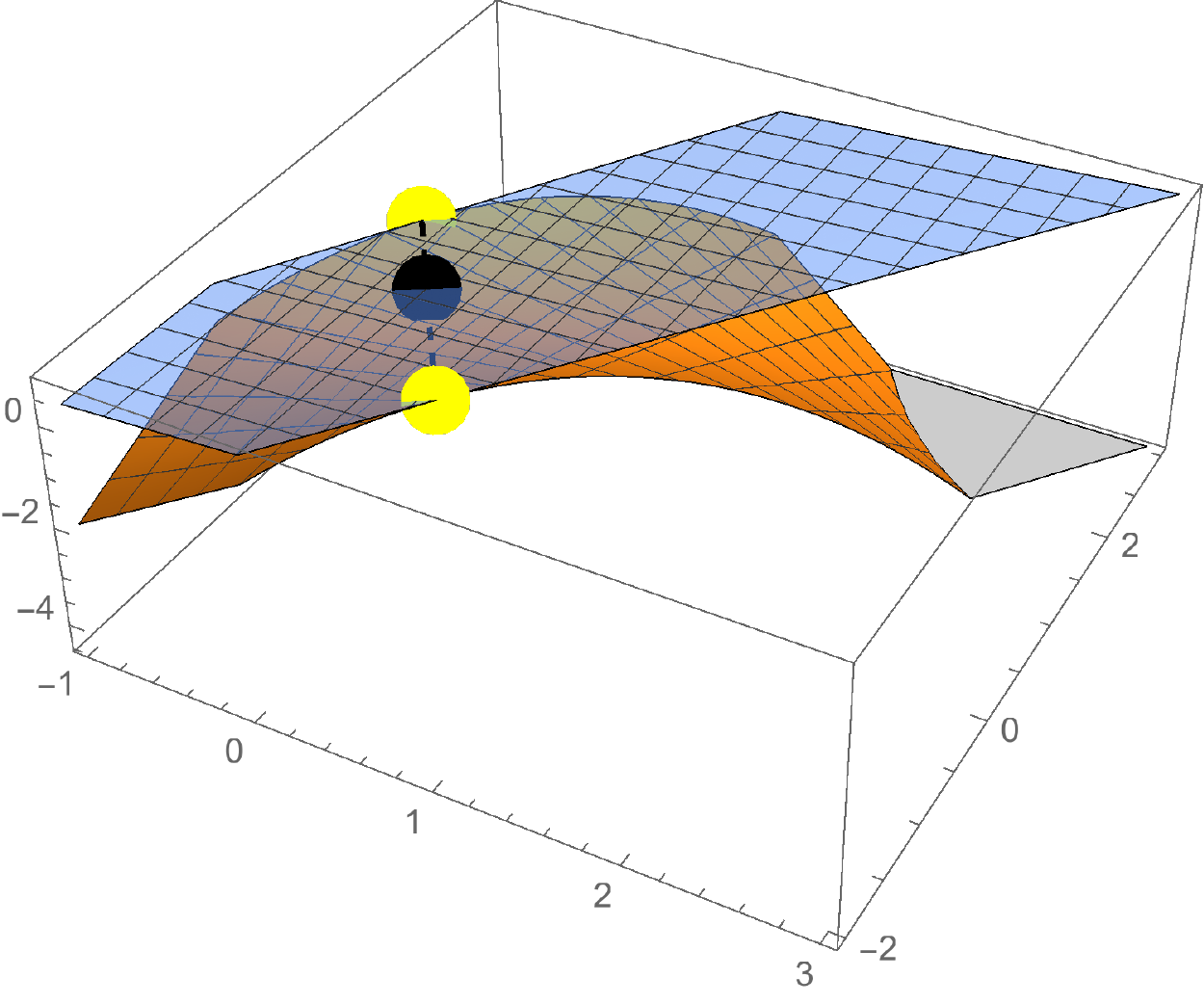}
  \end{minipage}
  \caption{All possibilities for computing a tangent linear inequality (blue) for the convex and concave envelope of
    $\xi\xj$ on a two-dimensional polytope $P \subset \R^2$ at a given point (black).  The inequality is obtained by
    selecting at most three points (yellow) that are on the boundary of $P$ such that the given point is in the convex hull
    of the selected points (dashed lines).}
  \label{fig:locatelli}
\end{figure}

\section{Two-dimensional projected relaxations}
\label{section:computeprojection}
Consider a single nonconvex quadratic constraint $\Xij = \xi\xj$ of~\eqref{eq:miqcpref} with $i \neq j$,
$\xi \in \ibox$, $\xj \in \jbox$, and $\Xij \in \R$. Let $\feasset$ be the set of feasible points of the original \MINLP~\eqref{eq:miqcpref}
and let
\begin{equation} \label{eq:proj:feasset}
  \projfeasset := \left\{ (\xi,\xj) \mid x \in \feasset \right\} \subseteq \ijbox
\end{equation}
be the projection of $\feasset$ onto the $(\xi,\xj)$-space.
The best possible convex relaxation for the nonconvex constraint $\Xij = \xi \xj$ is given by the convex hull of
$\{ (\xi,\xj,\Xij) \in \projfeasset \times \R \mid \xi \xj = \Xij\}$.
However, it is unclear how to enforce this relaxation in practice.
First, the set $\feasset$ is unknown and in general even finding a single point in $\feasset$ is \NP-hard.
Second, $\projfeasset$ can be a non-polyhedral, nonconvex, disconnected set and thus cannot be used by
Locatelli's algorithm. Hence, instead of targeting $\projfeasset$ directly, we propose to compute a polyhedral relaxation
$\Pij \subset \R^2$ of $\projfeasset$, i.e., $\projfeasset \subseteq \Pij$.
This relaxation is based on a polyhedral relaxation of $\feasset$, which we denote by
\begin{equation}\label{eq:lprelax}
    \relaxset := \{ (x,X) \mid  A_1 x + A_2 X \le d \},
\end{equation}
where $X$ is assumed to be a vector. These relaxations are readily available in LP-based spatial branch-and-bound algorithms.
They are constructed from linear constraints present in the original problem formulation, from cutting planes based on integrality information, and from other valid linearizations
of quadratic constraints such as gradient cuts.

Similar to~\eqref{eq:proj:feasset}, let
\begin{equation}
  \projrelaxset := \left\{ (\xi,\xj) \mid x \in \relaxset \right\} \subseteq \ijbox
\end{equation}
be the projection of $\relaxset$ onto the $(\xi,\xj)$-space. The best polyhedral relaxation from $\relaxset$ is $\projrelaxset$. Unfortunately, exponentially many inequalities may be necessary
to describe $\projrelaxset$~\cite{Ben2001}. For this reason, exact projection methods such as standard homotopy procedures~\cite{Nazareth1991}
may be overly expensive in practice.
This motivates the computation of a relaxation $\Pij$ of $\projrelaxset$. In view of
the complexity of Locatelli's algorithm, we would like for $\Pij$ to have few vertices and
facets. Specifically, we propose an algorithm that yields a $\Pij$ described by at most four axis-parallel and
at most four general inequalities. Later, we show that the quotient of the volume of $\projrelaxset$ and the volume of
the constructed $\Pij$ is bounded by $1/2$ from below.

\begin{remark}\normalfont
  An even tighter relaxation can be achieved by also discarding feasible points from the set $\feasset$ by using an
  objective cutoff $c^\T x \le \incumbent$. Typically, solutions with an objective value $\incumbent$ are found by
  heuristics during spatial branch-and-bound. Such a solution reduces the set of relevant feasible points to
  \begin{equation*}
    \feasset \cap \{(x,X) \mid c^\T x \le \incumbent\},
  \end{equation*}
  which might later result in even tighter $\Pij$.
\end{remark}

\subsection{Computing polyhedral projections with linear programming}
\label{section:projections:computing}

For $\projrelaxset \subsetneq \ijbox$ to hold, there must be at least one valid (facet-defining) inequality that
separates a vertex of $\ijbox$ from $\projrelaxset$.
To find some of those facets, if they exist, we follow a procedure akin to the shooting experiment~\cite{Hunsaker2008}.
The idea is to shoot a ray from a point $(\centerp_i, \centerp_j) \in \projrelaxset$ towards a vertex $(\xibar, \xjbar)$
of $\ijbox$.
This ray is going to intersect the boundary of $\projrelaxset$.
If the intersection is at the vertex, then the vertex is feasible for $\projrelaxset$.
Otherwise, any active constraint at the intersection point separates the vertex from $\projrelaxset$.
If the intersection point is in the interior of a facet, then that facet is the only active constraint. See
Figure~\ref{fig:inequality} for an illustration of the idea.
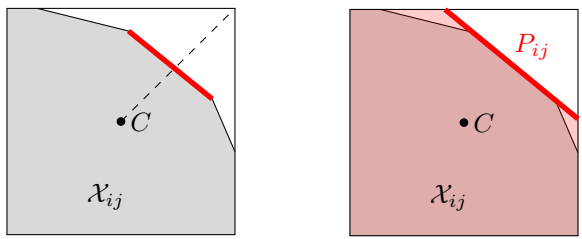
\begin{figure}[t]
    \centering
    \begin{minipage}{0.3\textwidth}
    \centering
    \begin{tikzpicture}
      \fill[color=black!15!white] (0,0) -- (0,3) -- (0.4,3) -- (1.6,2.7) -- (2.7, 1.8) -- (3,1.1) -- (3,0) -- cycle;
      \draw (0,0) -- (3,0) -- (3,3) -- (0,3) -- cycle;
      \draw (0.4,3) -- (1.6,2.7) -- (2.7, 1.8) -- (3,1.1);
      \draw[color=red, thick, line width=2] (1.6,2.7) -- (2.7, 1.8);
      \draw[fill=black] (1.5,1.5) circle (0.05) node[right]{$\centerp$};
      \draw[dashed] (1.5,1.5) -- (3,3);
      \node[] (t) at (1.3,0.5) {$\projrelaxset$};
    \end{tikzpicture}
  \end{minipage}
  \begin{minipage}{0.3\textwidth}
    \centering
    \begin{tikzpicture}
      \fill[color=black!15!white] (0,0) -- (0,3) -- (0.4,3) -- (1.6,2.7) -- (2.7, 1.8) -- (3,1.1) -- (3,0) -- cycle;
      \draw (0,0) -- (3,0) -- (3,3) -- (0,3) -- cycle;
      \draw (0.4,3) -- (1.6,2.7) -- (2.7, 1.8) -- (3,1.1);
      \draw[fill=red, opacity=0.2] (0,0) -- (0,3) -- (1.25,3.0) -- (3, 1.55) -- (3,0)-- cycle;
      \draw[fill=black] (1.5,1.5) circle (0.05) node[right]{$\centerp$};
      \node[] (t) at (1.3,0.5) {$\projrelaxset$};
      \node[] (p) at (2.4,2.5) {\textcolor{red}{$\Pij$}};
      \draw[color=red, thick, line width=2] (1.25,3.0) -- (3, 1.55);
    \end{tikzpicture}
  \end{minipage}
  \caption{An example for computing heuristically one facet of $\projrelaxset$. The idea is to find a facet with minimum distance
    with respect to the line connecting the center with a vertex of $\ijbox$. The red colored facet in the left picture
    is closest to the center among the three facets that separate the top-right vertex. The right picture shows that
    using this facet-defining inequality together with the bound constraints of $\xi$ and $\xj$ defines a polytope
    $\Pij$ which is a relaxation of $\projrelaxset$.}
  \label{fig:inequality}
\end{figure}
In our setting, the intersection point is $(\xi^*, \xj^*)$  where $(x^*, X^*, \theta^*)$ is the solution
of the following LP:
\begin{equation} \label{eq:diaglp_orig}
\begin{aligned}
               \max \; & \theta, \\
  \text{s.t. } & A_1 x + A_2 X \le d, \\
               & \centerp_i + \theta (\xibar - \centerp_i) = \xi, \\
               & \centerp_j + \theta (\xjbar - \centerp_j) = \xj, \\
               & \theta \in \R.
\end{aligned}
\end{equation}
Projecting out $\theta$ yields
\begin{equation} \label{eq:diaglp}
\begin{aligned}
  \max \; & \sign(\xibar - \centerp_i) \xi, \\
  \text{s.t. } & A_1 x + A_2 X \le d, \\
          & (\xj - \centerp_j)(\xibar - \centerp_i) = (\xi - \centerp_i)(\xjbar - \centerp_j),
\end{aligned}
\end{equation}
which is in the following denoted by $\diaglp{\xibar}{\xjbar}$.
As is shown next, the dual solution of this LP can be utilized to construct the inequality we are looking for.

Let $(x^*,X^*,\lambda^*,\mu^*)$ be an optimal primal-dual solution of $\diaglp{\xibar}{\xjbar}$, where $\lambda^* \geq 0$
are the dual multipliers of the inequality constraints and $\mu^* \in \R$ the dual multiplier for the equality
constraint of~\eqref{eq:diaglp}. Note that the aggregation
\begin{equation*}
  {\lambda^*}^\T (A_1 x + A_2 X) \le {\lambda^*}^\T d
\end{equation*}
is valid for $\relaxset$. Multiplying the stationarity condition
\begin{equation*}
  \sign(\xibar - \centerp_i) e_i^\T = {\lambda^*}^\T \mvec{A_1}{A_2} + \mu^* \left( \frac{\revision{e_i^\T}}{\xibar - \centerp_i} - \frac{\revision{e_j^\T}}{\xjbar - \centerp_j} \right)
\end{equation*}
of the Karush--Kuhn--Tucker~\cite{Karush2013,Kuhn2013} conditions by \revision{$(x^\T,X^\T)^\T$} shows that
\begin{equation*}
  \sign(\xibar - \centerp_i) \xi = {\lambda^*}^\T (A_1 x + A_2 X) + \mu^* \left( \frac{\xi}{\xibar - \centerp_i} - \frac{\xj}{\xjbar - \centerp_j} \right)
\end{equation*}
holds. Using $A_1 x + A_2 X \le d$ and reordering terms results in
\begin{equation} \label{eq:projineq}
  \left( \sign(\xibar - \centerp_i) - \frac{\mu^*}{\xibar - \centerp_i} \right) x_i + \frac{\mu^*}{\xibar - \centerp_i} x_j \le {\lambda^*}^\T d,
\end{equation}
which is valid for $\relaxset$ and only depends on $\xi$ and $\xj$ and is tight at the intersection point.

For having a complete method we need to specify $(\centerp_i, \centerp_j) \in \projrelaxset$.
The center of $\ijbox$ is guaranteed to be in $\projrelaxset$ after we applied \OBBT on $\xi$ and $\xj$ for the
relaxation $\relaxset$, as the next Lemma shows. Recall that \OBBT ensures $\lb_i = \min_{(x,X) \in \relaxset} \xi$ and
$\ub_i = \max_{(x,X) \in \relaxset} \xi$.
\begin{lemma}\normalfont\label{lemma:center}
  Let $\relaxset \neq \emptyset$, $\lb_i = \min_{(x,X) \in \relaxset} \xi$, $\ub_i = \max_{(x,X) \in \relaxset} \xi$,
  $\lb_j = \min_{(x,X) \in \relaxset} \xj$, and $\ub_j = \max_{(x,X) \in \relaxset} \xj$. Denote by
  \begin{equation*}
    \centerp := \left(\frac{\lb_i + \ub_i}{2}, \frac{\lb_j + \ub_j}{2}\right)
  \end{equation*}
  the center of $\ijbox$. It holds that $\centerp \in \projrelaxset$.
\end{lemma}
\begin{proof}\normalfont
  Assume that $C \not\in \projrelaxset$. It follows that there is an inequality
  $\alpha_i \xi + \alpha_j \xj \le \alpha_0$ that is valid for $\projrelaxset$ and separates $\centerp$. The center
  $\centerp$ can only be separated if the inequality separates at least two adjacent vertices of the rectangular domain
  $\ijbox$. Assume that it separates $(\lb_i,\lb_j)$ and $(\ub_i, \lb_j)$ (all other three cases work analogously),
  i.e., $\alpha_i \lb_i + \alpha_j \lb_j > \alpha_0$ and $\alpha_i \ub_i + \alpha_j \lb_j > \alpha_0$.  This immediately
  shows that there is no feasible point in $\projrelaxset$ with $\xj = \lb_j$, which is a contradiction to
  $\lb_j = \min_{(x,X) \in \relaxset} \xj$ and $\relaxset \neq \emptyset$. Figure~\ref{fig:lemma:center} illustrates the
  idea of the proof.
\end{proof}

\begin{figure}[t]
  \centering
  \hspace{20ex}
  \begin{tikzpicture}
    \fill[color=black!15!white] (0.0,0.95) -- (3,2.35) -- (3,3) -- (0,3) -- cycle;
    \draw (0,0.0) node[below]{$(\lb_i,\lb_j)$} -- (3,0.0) node[below]{$(\ub_i,\lb_j)$} -- (3,3) -- (0,3) --cycle;
    \draw[fill=black] (1.5,1.5) circle (0.05) node[right]{$\centerp$};
    \draw (-0.3,0.8) -- (3.3,2.5) node[right]{$\alpha_i \xi + \alpha_j \xj \le \alpha_0$};
    \draw[dashed,thick] (0,0.95) -- (3,0.95);
  \end{tikzpicture}
  \caption{Idea of the proof of Lemma~\ref{lemma:center}. The gray shaded area is the feasible region described by the
    inequality $\alpha_i \xi + \alpha_j \xj \le \alpha_0$. The dashed line corresponds to a tighter lower bound on $\xj$
    that is implied by the inequality.}
  \label{fig:lemma:center}
\end{figure}
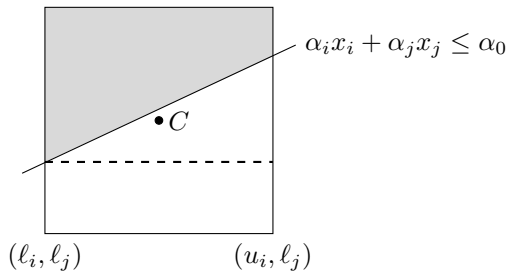

Lemma~\ref{lemma:center} implies that each inequality that is valid for $\projrelaxset$ can separate at
most one vertex of $\ijbox$ \revision{directly after \OBBT has been applied to $\xi$ and $\xj$}.
\revision{However, if tighter bounds from \OBBT are used to strengthen
  the linear relaxation $\relaxset$ further by, e.g., computing tighter convexifications for nonconvex constraints or propagating
  variables bounds via \FBBT, then the conditions in Lemma~\ref{lemma:center} may not be met anymore. For this reason, we solve~\eqref{eq:diaglp}
  immediately after \OBBT.}

Finally, we are able to define the polytope $\Pij$ by using the variable bounds $\ijbox$ and the derived
inequalities~\eqref{eq:projineq} for four choices of $\xibar$ and $\xjbar$, namely the four vertices of
$\ijbox$. Defining $\Pij$ like this has the advantage that it is described by at most eight inequalities and covers at
least half of the volume of $\projrelaxset$ as it is shown in the following section.

\begin{remark}\normalfont
  The problem of computing a facet-defining inequality for a projection of a polyhedron has been extensively studied in
  the literature.
  It corresponds to the ``project'' step in lift-and-project cuts~\cite{Balas1993,Balas2005}.
  The dual of \eqref{eq:diaglp_orig} is
  \begin{equation} \label{eq:duallp}
  \begin{aligned}
    \max \; & \beta - \alpha_i \centerp_i - \alpha_j \centerp_j, \\
    \text{s.t. } & \alpha_i e^\T_i + \alpha_j e^\T_j = \lambda^\T A_1, \\
            & 0 = \lambda^\T A_2, \\
            & \beta = \lambda^\T d, \\
            & \alpha_i(\xibar - \centerp_i) + \alpha_i(\xibar - \centerp_i) = 1, \\
            & \lambda \ge 0,
  \end{aligned}
  \end{equation}
  which can be interpreted as a cut generating linear program (CGLP) with the objective function of the so-called
  reverse polar CGLP~\cite[Chap.~2]{Serra2018} and the normalization constraint of Balas and
  Perregaard~\cite{Balas2002}. We refer to the thesis of Serra~\cite[Chap.~2]{Serra2018} for more details.
\end{remark}

\subsection{Volume bound}

We are interested in how much we lose by not computing the exact projection of the polyhedral relaxation $\relaxset$.
In the literature, the volume has been used as a measure for the strength of relaxations,
see~\cite{Lee1994,Speakman2017}. Following this line of thought, we provide a lower bound on the quotient of the volume
of $\projrelaxset$ and $\Pij$.

\begin{theorem}\label{theorem:volume}\normalfont
  \revision{
  Let $\relaxset$ be a relaxation of~\ref{eq:miqcpref} with $\lbi = \min \{\xi \mid (x,X) \in \relaxset \}$, $\ubi = \max \{\xi \mid (x,X) \in \relaxset \}$,
  $\lbj = \min \{\xj \mid (x,X) \in \relaxset \}$, $\ubj = \max \{\xj \mid (x,X) \in \relaxset \}$ for two variable indices $i,j \in \varindex$.
  Let $\projrelaxset = \{ (\xi,\xj) \mid (x,X) \in \relaxset \}$ be the two-dimensional projection of $\relaxset$ onto the $(\xi,\xj)$-space. Let
  $\Pij$ be a polytope that is given by the intersection of $\ijbox$ and~\eqref{eq:projineq} for the four choices
  $(\xibar,\xjbar) \in \{\lbi,\ubi\} \times \{\lbj,\ubj\}$. Then, the inequality
  \begin{equation*}
    \frac{\volume{\projrelaxset}}{\volume{\Pij}} \ge \frac{1}{2}
  \end{equation*}
  holds and the constant is best possible.
  }
\end{theorem}

\begin{proof}
  Since the volume quotient is invariant with respect to scaling and translating, we assume that all variable bounds are
  $[0,1]$. By construction, $\Pij$ is a relaxation of $\projrelaxset$. Because the conditions of Lemma~\ref{lemma:center} are
  met, it follows that the center point $C = (1/2,1/2)$ belongs to $\projrelaxset$ and thus also to $\Pij$. Let $p^k \in \R^2$
  for $k \in \{1,2,3,4\}$ be the four intersection points between $\Pij$ and the line connecting the center $C$. By construction,
  these four points belong to the set $\projrelaxset$.

  First, we construct an example that shows that \revision{the constant is best possible.} Let $(0,0)$, $(0,1-a)$, $(a,1)$, $(1,1)$,
  $(1,a)$, and $(1-a,0)$ be the vertices of $\Pij$ and $p^1 = (0,0)$, $p^2 = (a/2,1-a/2)$, $p^3 = (1,1)$, $p^4 =(1-a/2,a/2)$
  the vertices of $\projrelaxset$ depending on a parameter $a \in [0,1]$. See Figure~\ref{fig:volume:sharp} for an illustration
  of the construction. It follows that $\volume{\Pij} = 1 - a^2$ and $\volume{\projrelaxset} = 1 - a$ holds. As a consequence,
  \begin{equation*}
    \frac{\volume{\projrelaxset}}{\volume{\Pij}} = \frac{1-a}{1-a^2} = \frac{1-a}{(1-a)(1+a)} = \frac{1}{1+a}
  \end{equation*}
  converges to $\frac{1}{2}$ for $a \rightarrow 1$. Note that for $a=1$ the volume quotient exists but the polytopes
  $\Pij$ and $\projrelaxset$ reduce to a single line.

  \revisionstart
  Now, we prove the inequality. Since $\projrelaxset$ is a subset of $\Pij$, it immediately follows that
  \begin{equation*}
    \volume{\Pij} = \volume{\projrelaxset} + \volume{\Pij \backslash \projrelaxset} \, .
  \end{equation*}
  The inequality $\volume{\Pij \backslash \projrelaxset} \le \volume{\projrelaxset}$ is enough to show
  \begin{equation*}
    \volume{\Pij} = \volume{\projrelaxset} + \volume{\Pij \backslash \projrelaxset} \le 2 \cdot \volume{\projrelaxset} \, ,
  \end{equation*}
  which proves the theorem. We still need to prove the following claim.
  \begin{claim} \normalfont \label{theorem:volume:claim}
    $\volume{\Pij \backslash \projrelaxset} \le \volume{\projrelaxset}$
  \end{claim}
  \begin{proof}
    Let $q^k \in \R^2$ for $k \in \{1,2,3,4\}$ be four points in $\projrelaxset$ such that each point touches a different
    side of the $\ijbox$ box. The left picture in Figure~\ref{fig:claim:proof} shows how the points are labeled. The set
    \begin{equation*}
      \Xij' := \convexhull \{q^1,p^1,q^2,p^2,q^3,p^3,q^4,p^4\}
    \end{equation*}
    is by construction a subset of $\Xij$. As $\volume{\Pij \backslash \Xij'} \le \volume{\Pij \backslash \Xij}$ and
    $\volume{\Xij'} \le \volume{\Xij}$, showing the claim for $\Xij'$ implies the result for $\Xij$.

    The set $\Xij'$ decomposes into the four regions
    \begin{equation*}
      R_k := \convexhull\{C,q^k,p^k,q^{k+1}\} \subseteq \Xij'
    \end{equation*}
    for $k \in \{1,2,3,4\}$, whereas $q^5 = q^1$. The set $\Pij \backslash \Xij$ decomposes
    into eight triangles that are adjacent to the regions, see the right picture of Figure~\ref{fig:claim:proof}. In
    the following, we show that the area of each $R_k$ is at least as big as the area of the two corresponding
    triangles, which proves the claim.

    Consider the region $R_1$ in the left-bottom corner. If $(a_1,0)^\T$ and $(0,b_1)^\T$ are the endpoints of the facet in
    $\Pij$ that contains $p^1 = (c,c)^\T$, then $c = a_1 \, b_1 / (a_1 + b_1)$. Note that the claim is true if $a_1 = 0$ or $b_1 = 0$
    because in this case the two adjacent triangles are empty. Let $q^1 = (a_2,0)^\T$ and $q^2 = (0,b_2)^\T$. The area of the triangle
    $\Delta_1 := \convexhull \{ (a_1,0)^\T, p^1, q^1\} \subseteq \Pij \backslash \Xij$ is
    \begin{equation*}
      \volume{\Delta_1} = \frac{c \, (a_2 - a_1)}{2}
    \end{equation*}
    and the area of the second triangle $\Delta_2 := \convexhull \{ (0,b_1)^\T, p^1, q^2\} \subseteq \Pij \backslash \Xij$ is
     \begin{equation*}
      \volume{\Delta_2} = \frac{c \, (b_2 - b_1)}{2} \, .
    \end{equation*}
    The area of the quadrilateral is given by the area of two triangles $\Delta_3 = \convexhull\{C,p^1,q^1\}$ and
    $\Delta_4 = \convexhull\{C,q^2,p^1\}$. Their areas are
    \begin{equation*}
      \volume{\Delta_3} = \frac{a_2}{4} - \frac{a_2 \, c}{2} = \frac{a_2 \,(1-2c)}{4}
    \end{equation*}
    and
    \begin{equation*}
      \volume{\Delta_4} = \frac{b_2}{4} - \frac{b_2 \, c}{2} = \frac{b_2 \,(1-2c)}{4} \, .
    \end{equation*}
    Finally, we show that the area of $\Delta_1$ and $\Delta_2$ is less or equal than the area of $\Delta_3$ and $\Delta_4$, which
    proves the claim. After algebraic manipulation, we get
    \begin{align}
      \volume{\Delta_1} + \volume{\Delta_2} &- \volume{\Delta_3} - \volume{\Delta_4} = \frac{c \, (a_2 - a_1)}{2} + \frac{c \, (b_2 - b_1)}{2} - \frac{a_2 \,(1-2c)}{4} - \frac{b_2 \,(1-2c)}{4} \nonumber \\
        &= \frac{-2 a_1^2 b_1 - 2 a_1 b_1^2 + (a_2 + b_2) (4 a_1 b_1-a_1-b_1)}{4 (a_1 + b_1)} \label{theorem:volume:claim:nominator}
    \end{align}
    where the second step used the definition of $c$, i.e., $c = a_1 b_1 / (a_1 + b_1)$. Since the denominator of~\eqref{theorem:volume:claim:nominator}
    is positive, showing that the nominator of~\eqref{theorem:volume:claim:nominator} is non-positive implies
    \begin{align*}
      \volume{\Delta_1} + \volume{\Delta_2} &- \volume{\Delta_3} - \volume{\Delta_4} \le 0 \, .
    \end{align*}
    We consider two cases.
    \begin{enumerate}
      \item[\textbf{Case 1:}] $4 a_1 b_1-a_1-b_1 \le 0$ \\
        The nominator of~\eqref{theorem:volume:claim:nominator} consists of three non-positive terms.
      \item[\textbf{Case 2:}] $4 a_1 b_1-a_1-b_1 > 0$ \\
        Since $a_2, b_2 \in [0,1]$, it follows that
        \begin{align*}
          -2 a_1^2 b_1 - 2 a_1 b_1^2 + (a_2 + b_2) (4 a_1 b_1-a_1-b_1) &\le -2 a_1^2 b_1 - 2 a_1 b_1^2 + 2 (4 a_1 b_1-a_1-b_1) \\
          &= 2 a_1 (2 b_1 - b_1^2 - 1) + 2 b_1 (2 a_1 - a_1^2 - 1) \\
          &= -2 a_1 (b_1 - 1)^2 - 2 b_1 (a_1 - 1)^2 \le 0 \, ,
        \end{align*}
        which proves that the nominator of~\eqref{theorem:volume:claim:nominator} is non-positive.
    \end{enumerate}
  \end{proof}
  \revisionend

\end{proof}

\begin{figure}[t]
  \centering
  \def\aval{0.5}
  \begin{tikzpicture}[scale=3]
    \draw (0,0) rectangle (1,1);

    \draw[fill=red,opacity=0.2] (0,0) -- (0,1-\aval) -- (\aval,1) -- (1,1) -- (1,\aval) -- (1-\aval,0) -- cycle;

    \draw[fill=black,opacity=0.2] (0,0) -- (0.5*\aval,1-0.5*\aval) -- (1,1) -- (1-0.5*\aval,0.5*\aval) -- cycle;

    \draw[dashed] (0,0) -- (1,1);
    \draw[dashed] (0,1) -- (1,0);

    \draw[fill=black](0.5,0.5) node[left]{$\centerp$} circle(0.5pt);

    \draw[fill=black](0,0) circle(0.5pt) node[below]{\small $p^1$};
    \draw[fill=black](0.5*\aval,1-0.5*\aval) circle(0.5pt) node[left]{\small $p^2$};
    \draw[fill=black](1,1) circle(0.5pt) node[above]{\small $p^3$};
    \draw[fill=black](1-0.5*\aval,0.5*\aval) circle(0.5pt) node[right]{\small $p^4$};

    \draw[fill=black](0,1-\aval) circle(0.5pt) node[left]{\tiny $\mvec{0}{1-a}$};
    \draw[fill=black](\aval,1) circle(0.5pt) node[above]{\tiny $\mvec{a}{1}$};
    \draw[fill=black](1,\aval) circle(0.5pt) node[right]{\tiny $\mvec{1}{a}$};
    \draw[fill=black](1-\aval,0) circle(0.5pt) node[below]{\tiny $\mvec{1-a}{0}$};

  \end{tikzpicture}
  \caption{\revision{Construction of a parametric example that shows that $\frac{\volume{\projrelaxset}}{\volume{\Pij}}$
    approaches $\frac{1}{2}$ when $a$ approaches $1$.} The gray region is $\projrelaxset$ and the red region is $\Pij$.}
  \label{fig:volume:sharp}
\end{figure}
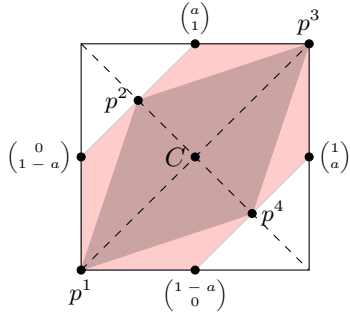

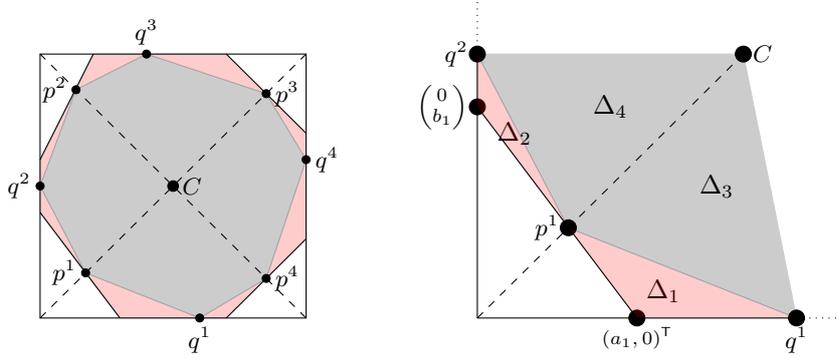
\begin{figure}
  \centering
  \begin{minipage}[t]{0.4\textwidth}
    \centering
    \begin{tikzpicture}[scale=3.5]

      \draw[fill=black,opacity=0.2] (0,0.5) -- (0.135,0.865) -- (0.4,1) -- (0.85,0.85) -- (1,0.6) -- (0.85,0.15) -- (0.6,0) -- (0.171,0.171) -- cycle;

      \draw[fill=red,opacity=0.2] (0.171,0.171) -- (0,0.5) -- (0,0.4) -- cycle;
      \draw[fill=red,opacity=0.2] (0,0.5) -- (0.135,0.865) -- (0,0.6) -- cycle;
      \draw[fill=red,opacity=0.2] (0.4,1) -- (0.2,1) -- (0.135,0.865) -- cycle;
      \draw[fill=red,opacity=0.2] (0.4,1) -- (0.85,0.85) -- (0.7,1) -- cycle;
      \draw[fill=red,opacity=0.2] (0.85,0.85) -- (1,0.6) -- (1,0.7) -- cycle;
      \draw[fill=red,opacity=0.2] (1,0.6) -- (1,0.3) -- (0.85,0.15) -- cycle;
      \draw[fill=red,opacity=0.2] (0.85,0.15) -- (0.6,0) -- (0.7,0) -- cycle;
      \draw[fill=red,opacity=0.2] (0.6,0) -- (0.3,0) -- (0.171,0.171) -- cycle;

      \draw (0,0) -- (0,1) -- (1,1) -- (1,0) -- cycle;

      \draw[dashed] (0,0) -- (1,1);
      \draw[dashed] (1,0) -- (0,1);
      \draw[fill=black] (0.5,0.5) circle (0.02) node[right]{\small $C$};

      \draw (0,0.4) -- (0.3,0);
      \draw[fill=black] (0.171,0.171) circle (0.015) node[left]{\small $p^1$};

      \draw (0.7,0) -- (1,0.3);
      \draw[fill=black] (0.85,0.15) circle (0.015) node[right]{\small $p^4$};

      \draw (0.7,1) -- (1,0.7);
      \draw[fill=black] (0.85,0.85) circle (0.015) node[right]{\small $p^3$};

      \draw (0,0.6) -- (0.2,1);
      \draw[fill=black] (0.135,0.865) circle (0.015) node[left]{\small $p^2$};

      \draw[fill=black] (0.6,0) circle (0.015) node[below]{\small $q^1$};
      \draw[fill=black] (0,0.5) circle (0.015) node[left]{\small $q^2$};
      \draw[fill=black] (0.4,1) circle (0.015) node[above]{\small $q^3$};
      \draw[fill=black] (1,0.6) circle (0.015) node[right]{\small $q^4$};

    \end{tikzpicture}
  \end{minipage}
  \begin{minipage}[t]{0.4\textwidth}
    \centering
    \begin{tikzpicture}[scale=7]
      \draw (0,0) -- (0.6, 0);
      \draw[dotted] (0.6,0) -- (0.7,0);
      \draw (0,0) -- (0, 0.5);
      \draw[dotted] (0,0.5) -- (0,0.6);
      \draw[dashed] (0,0) -- (0.5,0.5);

      \node[] (D1) at (0.35,0.05) {$\Delta_1$};
      \node[] (D2) at (0.07,0.35) {$\Delta_2$};
      \node[] (D3) at (0.45,0.25) {$\Delta_3$};
      \node[] (D4) at (0.25,0.4) {$\Delta_4$};

      \draw[fill=black] (0.5,0.5) circle (0.015) node[right]{\small $C$};
      \draw[fill=black] (0.6,0) circle (0.015) node[below]{\small $q^1$};
      \draw[fill=black] (0,0.5) circle (0.015) node[left]{\small $q^2$};
      \draw (0,0.4) -- (0.3,0);
      \draw[fill=black] (0.3,0) circle (0.015) node[below]{\scriptsize $(a_1,0)^\T$};
      \draw[fill=black] (0,0.4) circle (0.015) node[left]{\scriptsize $\mvec{0}{b_1}$};
      \draw[fill=black] (0.171,0.171) circle (0.015) node[left]{\small $p^1$};
      \draw[fill=black,opacity=0.2] (0.5,0.5) -- (0.6,0) -- (0.171,0.171) -- (0,0.5) -- cycle;
      \draw[fill=red,opacity=0.2] (0.171,0.171) -- (0,0.5) -- (0,0.4) -- cycle;
      \draw[fill=red,opacity=0.2] (0.6,0) -- (0.3,0) -- (0.171,0.171) -- cycle;
    \end{tikzpicture}
  \end{minipage}
  \caption{\revision{Construction of $\Xij'$ in the proof of Claim~\ref{theorem:volume:claim}. The inner polytope is $\Xij'$ (gray) is inscribed in
  $\Pij$. The main idea of the proof is that $\Pij \backslash \Xij$ (red) has smaller area than $\Xij$. The right picture illustrates that
  the region $R_1 = \Delta_3 \cup \Delta_4$ is larger than the two adjacent triangles $\Delta_1$ and $\Delta_2$.}}
  \label{fig:claim:proof}
\end{figure}

\begin{remark}\normalfont \label{remark:volume}
  \revision{The construction of the parametric example in the proof of Theorem~\ref{theorem:volume} requires that $\Pij$ contains
  two facets that are not axis-parallel.} If only one facet of $\Pij$ is not axis-parallel, the volume quotient is bounded by
  $\frac{2+\sqrt{2}}{4} \approx 0.85$.
\end{remark}

Theorem~\ref{theorem:volume} and Remark~\ref{remark:volume} provide some
theoretical justification \revision{why it suffices to compute} a relaxation of
the projection. From a practical point of view, spending more time in computing $\projrelaxset$ exactly might not pay off because we are only
projecting a relaxation of the feasible region.

\subsection{Computational aspects}

So far, we have only considered a single term $\xi\xj$, but in general~\eqref{eq:miqcpref} contains up to $O(\nvars^2)$ many
bilinear terms. With growing number of variables, it may become computationally too expensive to solve~\eqref{eq:diaglp} for all bilinear terms.
In order to save unnecessary solves of $\diaglp{\xibar, \xjbar}$, we observe the following: \revision{The existence of a feasible solution
$(x^*,X^*) \in \relaxset$ in which $(\xi^*,\xj^*)$ is} a vertex of $\ijbox$ proves that no useful inequality for $\Pij$ can be
found that cuts off $(\xi^*,\xj^*)$. This observation is similar to the bound filtering in the branch-and-contract
algorithm~\cite{Zamora1999} and can be exploited as an aggressive filtering strategy, as it has been done in
\OBBT~\cite{Gleixner2017}.
\revision{The idea of bound filtering is to use a solution $(x^*,X^*)$ of an \LP relaxation $\relaxset$ and to
  check for which variables $\xi$ the solution value $\xi^*$ is equal to $\lbi$
  or $\ubi$. If $\xi^* = \lbi$ ($\xi^* = \ubi$) holds then \OBBT cannot find a tighter
  lower (upper) bound for $\xi$. In addition to considering solutions from previous \OBBT \LPs, aggressive bound filtering
  solves auxiliary \LPs with an objective function $v^\T x$ for a vector $v \in \{-1,0,1\}^\nvars$ to push as many unfiltered variables
  as possible to their lower or upper bounds. We refer to~\cite{Gleixner2017} for more details.}

In the following, we present the generic Algorithm~\ref{algo:projections} that first applies \OBBT to ensure that the
center point $(\centerp_i,\centerp_j)$ belongs to $\projrelaxset$ and afterwards computes a relaxation of
$\projrelaxset$ as discussed above.

\begin{algorithm}[t]
  \caption{Two-dimensional projections}
  \label{algo:projections}
  \begin{algorithmic}[1]
  \REQUIRE{linear relaxation $\relaxset = \{(x,X) \mid A_1 x + A_2 X \le d\}$ of~\eqref{eq:miqcpref}}
  \ENSURE{a list $\mathcal{P}$ of two-dimensional polytopes for each bilinear term $\xi\xj$}
  \STATE{$K \leftarrow \{(i,j) \in \varindex \times \varindex \mid i < j \land \exists k \in \consindex: (Q_k)_{ij} \neq 0\}$} \label{algo:projections:terms} \COMMENT{collect bilinear terms}
  \STATE{$\mathcal{P} \leftarrow \emptyset$, \revision{$F \leftarrow \emptyset$}}
  \FOR[call \OBBT]{$i \in \varindex : \exists j \in \varindex$ such that $(i,j) \in K$}
    \STATE{$\ibox \leftarrow$ apply \OBBT on $\xi$; let $x^*$ be the \OBBT \LP solution} \label{algo:projections:obbt}
    \STATE{\revision{$F \leftarrow F \, \cup \ \left\{ (i',j',x_{i'}^*,x_{j'}^*) \mid (i',j') \in K \land x_{i'}^* \in \{\lb_{i'},\ub_{i'}\} \land x_{j'}^* \in \{\lb_{j'},\ub_{j'}\} \right\}$}} \label{algo:projections:filter1}
  \ENDFOR
  \FOR{$(i,j) \in K$}
    \STATE{$\Pij \leftarrow \ijbox$}
    \FOR[iterate through all vertices]{$(\xibar, \xjbar) \in \{\lb_i,\ub_i\} \times \{\lb_j,\ub_j\}$}
      \IF[consider unfiltered candidates]{$(i,j,\xibar,\xjbar) \not\in F$} \label{algo:projections:unfiltered}
      \STATE{$(x^*,X^*,\lambda^*,\mu^*) \leftarrow $ solve $\diaglp{\xibar}{\xjbar}$} \label{algo:projections:lp}
        \STATE{\revision{$F \leftarrow F \, \cup \ \left\{ (i',j',x_{i'}^*,x_{j'}^*) \mid (i',j') \in K \land x_{i'}^* \in \{\lb_{i'},\ub_{i'}\} \land x_{j'}^* \in \{\lb_{j'},\ub_{j'}\} \right\}$}} \label{algo:projections:filter2}
        \IF[create linear inequality]{$\xi^* \neq \xibar \land \xj^* \neq \xjbar$}
          \STATE{extract valid inequality~\eqref{eq:projineq} from dual solution $(\lambda^*,\mu^*)$} \label{algo:projections:inequality}
          \STATE{\revision{$\Pij \leftarrow \Pij \, \cap \, \{ (\xi,\xj) \mid \eqref{eq:projineq} \text{ holds} \}$}}
        \ENDIF
      \ENDIF
    \ENDFOR
    \STATE{\revision{add $\Pij$ to $\mathcal{P}$}}
  \ENDFOR
  \RETURN{$\mathcal{P}$}
  \end{algorithmic}
\end{algorithm}

In Line~\ref{algo:projections:terms}, Algorithm~\ref{algo:projections} computes an index set for all occurring bilinear
terms. \OBBT is called in Line~\ref{algo:projections:obbt} for each variable that appears bilinearly to ensure that the
requirements of Lemma~\ref{lemma:center} are met. Afterward, in Line~\ref{algo:projections:lp}, for each term $\xi\xj$
the algorithm considers all vertices $(\xibar,\xjbar)$ of $\ijbox$ and solves $\diaglp{\xibar}{\xjbar}$. The result is a
primal-dual optimal solution $(x^*,X^*,\lambda^*,\mu^*)$ that is used in Line~\ref{algo:projections:inequality} for
generating a valid inequality for $\projrelaxset$. \revision{The \LP solutions from Line~\ref{algo:projections:obbt}
and~\ref{algo:projections:lp} are used to update the set of filtered candidates $F$ in Line~\ref{algo:projections:filter1}
and~\ref{algo:projections:filter2}. In Line~\ref{algo:projections:unfiltered}, a candidate $(\xi,\xj)$ for the direction
$(\xibar,\xjbar) \in \{\lb_i,\ub_i\} \times \{\lb_j,\ub_j\}$ is only considered if $(i,j,\xibar,\xjbar)$ has not been filtered out.}

In our implementation, all bilinear terms are ordered by how often they appear in different constraints of the original
\MIQCP. As a tie-break, we use the term $\xi\xj$ for which the volume of $\ijbox$, i.e.,
$(\ub_i - \lb_i)(\ub_j - \lb_j)$ is maximized.

Algorithm~\ref{algo:projections} could either solve $\diaglp{\xibar}{\xjbar}$ or its dual formulation~\eqref{eq:duallp} for
deriving the two-dimensional projections. However, solving $\diaglp{\xibar}{\xjbar}$ has two technical advantages:
\begin{enumerate}
\item The linear relaxation $\relaxset$ is available in \LP-based spatial branch-and-bound solvers and only needs to be
  extended by a single linear equality constraint for solving $\diaglp{\xibar}{\xjbar}$. This is beneficial compared to
  constructing~\eqref{eq:duallp} for a relaxation that contains many variables and constraints.
\item Due to the close connection to \OBBT, it is possible to warm start from a previously computed basis of an
  \OBBT-\LP. This would require to restructure Algorithm~\ref{algo:projections} in a way that it solves
  $\diaglp{\xibar}{\xjbar}$ after the bounds of $\xi$ and $\xj$ have been tightened by \OBBT. \revision{However,
  restoring a previous \LP basis causes a significant overhead that cannot be compensated by the warm start capabilities
  of the \LP solver. For this reason, our implementation of Algorithm~\ref{algo:projections} does not utilize a previously
  computed \LP basis.}
\end{enumerate}

After computing inequalities of the form~\eqref{eq:projineq}, we apply Locatelli's algorithm to strengthen the linear
relaxation of $\Xij = \xi\xj$ through separation during the tree search. Moreover, the computed $\Pij$ can not
only be used to improve separation but also to strengthen variable bounds of $\Xij$, $\xi$, and $\xj$, as shown in the
next section.

\section{Using 2D projections for propagation}
\label{section:propagation}

Tight variable bounds are crucial when computing linear (or convex) relaxations for \MIQCPs during spatial
branch-and-bound. Stronger bounds on $\xi$, $\xj$, and $\Xij$ not only affect the relaxation of $\Xij = \xi\xj$ but also
other constraints that involve these variables, including linear constraints. Propagating these constraints in turn might lead to further bound
reductions of variables that appear in other nonconvex constraints~\cite{BelottiCafieriLeeLiberti2012TR,Puranik2017} and subsequently result in tighter relaxations.

In the following, we show how to use a two-dimensional projection $\Pij$ to derive tighter bounds on $\xi$, $\xj$, and
$\Xij$ by solving nonconvex optimization problems that can be efficiently solved.

\subsection{Forward propagation}

Given a polytope $\Pij $, the best possible lower/upper bound for $\Xij$ on $\Pij$ is given by
\begin{equation} \label{eq:propagation:polytope}
  \min / \max \{ \Xij \mid \Xij = \xi \xj, \, (\xi,\xj) \in \Pij \},
\end{equation}
which is a nonconvex optimization problem. Denote by $F(\Pij)$ the facets of $\Pij$, and let
\begin{equation*}
  C(\Pij) := \left\{ \argmax_{(\xi,\xj) \in F} \xi\xj \mid F \in F(\Pij) \right\} \cup \left\{ \argmin_{(\xi,\xj) \in F} \xi\xj \mid F \in F(\Pij) \right\}
\end{equation*}
be the set of \emph{optimal points} for maximizing and minimizing $\xi \xj$ over each facet of $\Pij$.
For example, if $F = \{ (\xi,\xj) \in \ijbox \mid a_i \xi + a_j \xj = a_0 \}$ is a facet of $\Pij$ with $a_i \neq 0$ and
$a_j \neq 0$, then $\xi\xj$ restricted to $F$ is $-\frac{a_i}{a_j} \xi^2 + \frac{a_0}{a_j} \xi$. The critical point of
this function is $\frac{a_0}{2a_i}$. Thus $(\frac{a_0}{2a_i},\frac{a_0}{2a_j}) \in C(\Pij)$ if and only if
$\frac{a_0}{2a_j} \in \varbounds{i}$ and $\frac{a_0}{2a_i} \in \varbounds{j}$. Otherwise, both vertices of $F$ belong to
$C(\Pij)$.

See Figure~\ref{fig:propagation:inteval}
for an illustration of the points $C(\Pij)$. The following theorem shows that~\eqref{eq:propagation:polytope} can be
solved by computing the minimum / maximum on the discrete set $C(\Pij)$.

\begin{figure}[t]
  \centering
  \includegraphics[width=0.5\textwidth]{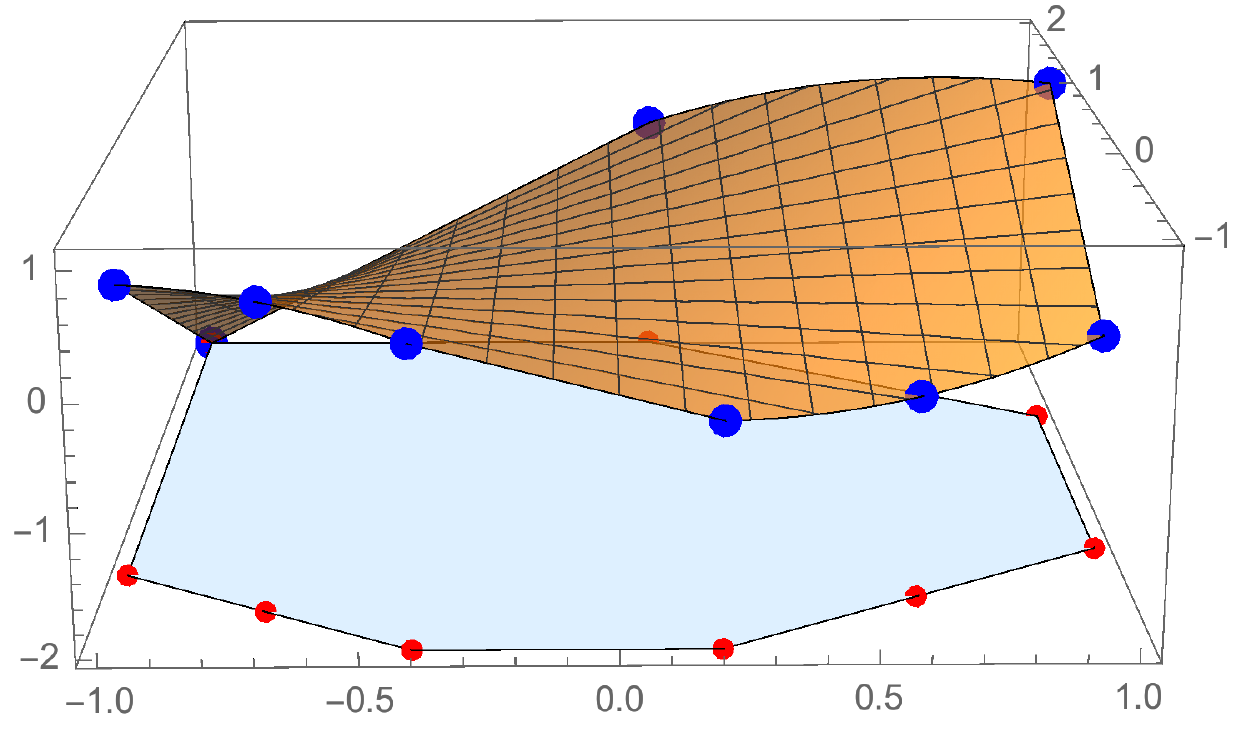}
  \caption{An example of how to compute the minimum and maximum of $\xi \xj$ on $\Pij \subset [-1,1] \times [-1,2]$. The
    red points are the points in $C(\Pij)$.}
  \label{fig:propagation:inteval}
\end{figure}

\begin{theorem}\normalfont\label{theorem:propagation:Xij}
  Let $\Pij \subset \R^2$ be a polytope and let $C(\Pij)$ be the optimal points of $\Pij$. Then the equality
  \begin{equation*}
    \min \{ \alpha \xi \xj \mid (\xi,\xj) \in \Pij \} = \min \{ \alpha \xi \xj \mid (\xi,\xj) \in C(\Pij) \}
  \end{equation*}
  holds for $\alpha \in \{-1,1\}$.
\end{theorem}

\begin{proof}
  First, due to the fact that $\xi\xj$ is bilinear, the minimum and maximum must be attained at the boundary of
  $\Pij$. Restricted to a facet, $\xi \xj$ achieves its maximum and minimum at a point in $C(\Pij)$.
\end{proof}

By construction, $\Pij$ has at most four facets that are not axis-parallel. This bounds the number of points in
$C(\Pij)$ by 12. Computing these points requires only simple algebraic computations as illustrated in the example above.

\subsection{Reverse propagation}

There are two ways to obtain tighter variable bounds for $\xi$ and~$\xj$ by utilizing $\Pij$. First, after branching on
$\xi$ or $\xj$, it is possible that a facet of $\Pij$ cuts off two vertices of the rectangular domain for some locally
valid bounds. \revision{This implies that at least one of the variable bounds of $\xi$ or $\xj$ can be tightened.}

Second, the bounds of $\Xij$ define a level set for the bilinear term $\xi \xj$. Intersecting the level set with $\Pij$
might imply tighter lower and upper bounds on $\xi$ and $\xj$. Even though the intersection is in general a nonconvex
region, we show that the best possible variable bounds that are implied by the intersection can be computed by
considering a finite set of points.

In the following, we give more details on the two possible types of bound reductions for $\xi$ and $\xj$.

\paragraph{Branching reductions}

Even though the facets of $\Pij$ are valid and redundant inequalities for the relaxation $\relaxset$ that has been used
for computing $\Pij$, they are still useful for deriving bound reductions on $\xi$ and $\xj$ during the tree
search. Figure~\ref{fig:propagation:branching} shows that $\Pij$ implies tighter bounds on $\xj$ after branching on
$\xi$. Note that optimizing $\pm \xj$ over $\relaxset$ leads to bounds that are at least as tight as the bounds implied
by $\Pij$. However, finding these bounds either requires solving an expensive \OBBT-LP or propagating several linear
constraints with \FBBT. The strength of using $\Pij$ together with variable bound changes due to branching is that the
facets of $\Pij$ contain information of multiple inequalities of $\relaxset$ and are computationally cheap to
propagate. Using the facets of $\Pij$ in this fashion is very similar to the so-called Lagrangian Variable Bounds of
Gleixner et al~\cite{GleixnerWeltge2013}, which are aggregations of linear constraints that are learned during \OBBT and
used as a computationally cheap approximation for \OBBT during the tree search.

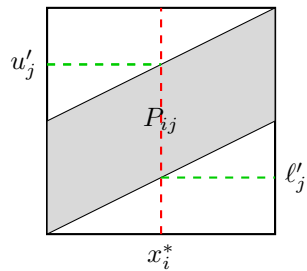
\begin{figure}[t]
  \centering
  \begin{tikzpicture}
    \draw[thick] (0,0) rectangle (3,3);
    \draw[fill=black!15!white] (0,0) -- (0,1.5) -- (3,3) -- (3,1.5) -- (0,0);
    \node[] (xi) at (1.5,-0.3){$\xi^*$};
    \node[] (Pij) at (1.5,1.5){$\Pij$};

    \draw[thick,dashed,color=red] (1.5,0) -- (1.5,3);
    \draw[thick,dashed,color=green!80!black] (0,2.25) node[left,color=black]{$\ub_j'$} -- (1.5,2.25);
    \draw[thick,dashed,color=green!80!black] (1.5,0.75) -- (3.0,0.75) node[right,color=black]{$\lb_j'$};
    \draw[thick,dashed,color=green!80!black] (0,2.25) -- (1.5,2.25);
  \end{tikzpicture}
  \caption{
    \revision{The example shows that the upper bound of $\xj$ can be improved after using $\xi^*$ as branching
    point for $\xi$.} The vertical red line is the branching point of $\xi$. The horizontal green lines correspond to a tighter lower bound
    $\lb_j'$ and a tighter upper bound $\ub_j'$ in both subproblems, respectively.}
  \label{fig:propagation:branching}
\end{figure}

\paragraph{Level set reductions}

Let $[\lb_{ij},\ub_{ij}]$ be bounds on $\Xij$ such that
\begin{equation*}
  \lb_{ij} > \min \{ \xi\xj \mid (\xi,\xj) \in \Pij \}
\end{equation*}
or
\begin{equation*}
  \ub_{ij} < \max \{ \xi\xj \mid (\xi,\xj) \in \Pij \}
\end{equation*}
holds. This means that the bounds on $\Xij$ are not implied by $\xi\xj$ on $\Pij$. The best possible lower/upper bound
on $\xi$ (and analogous for $\xj$) using $\Pij$ and the bounds $[\lb_{ij},\ub_{ij}]$ is given by
\begin{equation}\label{eq:propagation:levelset}
  \min / \max \left\{ \xi \mid \lb_{ij} \le \xi \xj \le \ub_{ij} , \, (\xi,\xj) \in \Pij \right\},
\end{equation}
which is a nonconvex optimization problem. Figure~\ref{fig:reverseprop} illustrates that intersecting the level set of
$\xi\xj$ and $\Pij$ can imply stronger bounds on $\xi$ and $\xj$. Similar to Theorem~\ref{theorem:propagation:Xij}, we
show that~\eqref{eq:propagation:levelset} can be efficiently solved by scanning a finite set of points. Let
$IP \subseteq \Pij$ consist of the vertices of $\Pij$ that satisfy $\xi\xj \in [\lb_{ij},\ub_{ij}]$, and the
intersection points of each facet of $\Pij$ with $\{(\xi,\xj) \mid \xi\xj = \lb_{ij}\}$ or
$\{(\xi,\xj) \mid \xi\xj = \ub_{ij}\}$. In other words, $IP$ is the set of feasible points for which at least two
constraints of~\eqref{eq:propagation:levelset} are active.  Since the vertices \revision{and facets of $\Pij$} are explicitly given, computing points
in $IP$ reduces to solving a univariate quadratic equation.

The following theorem shows that it suffices to consider the points in $IP$ to solve~\eqref{eq:propagation:levelset}.

\begin{theorem}\normalfont\label{theorem:levelset}
  Let $\xi\xj$ a bilinear term, $[\lb_{ij},\ub_{ij}]$ bounds on $\xi\xj$, and $\Pij \subseteq \R^2$ a polytope. Then the equality
  \begin{equation*}
    \min \left\{ \alpha \xi \mid \lb_{ij} \le \xi \xj \le \ub_{ij} , \, (\xi,\xj) \in \Pij \right\} = \min \left\{ \alpha \xi \mid (\xi,\xj) \in IP \right\}
  \end{equation*}
  holds for $\alpha \in \{-1,1\}$.
\end{theorem}

\begin{proof}
    We only prove the theorem for the objective function $\xi$ since $-\xi$ is analogous.
    Let $\xi^*$ be the optimal value.
    As the objective function is linear, every optimum is at the boundary.
    Therefore, at least one constraint it active.
    We will show that there is at least one optimum for which at least two constraints are active, i.e., is in $IP$.
    Let $(\xi^*, \xj^*)$ be any optimal point.

    If the only active constraint is linear, then it must be $x_i \geq \xi^*$.
    Since the feasible region is bounded, there is an $M > 0$ such that $(\xi^*, \xj^* + M)$ is infeasible.
    Therefore, for some $\xj \in [\xj^*, \xj^* + M]$, $(\xi^*, \xj)$ is active for at least two constraints.

    If the only active constraint is nonlinear, say, $\xi \xj = \phi$ with $\phi \in \{\lb_{ij},
    \ub_{ij}\}$, then the region $\{(\xi, \xj) : \xi \xj = \phi\}$, in a neighborhood of $(\xi^*, \xj^*)$, must be
    contained in $\xi \geq \xi^*$.
    This can only happen when \revision{$\xi^* = 0$} and the same argument as above shows that there is an $\xj$ such that $(\xi^*,
    \xj)$ is active for another constraint.
\end{proof}

\begin{figure}[t]
  \centering
  \begin{minipage}{0.3\textwidth}
    \centering
    \includegraphics[width=0.7\textwidth]{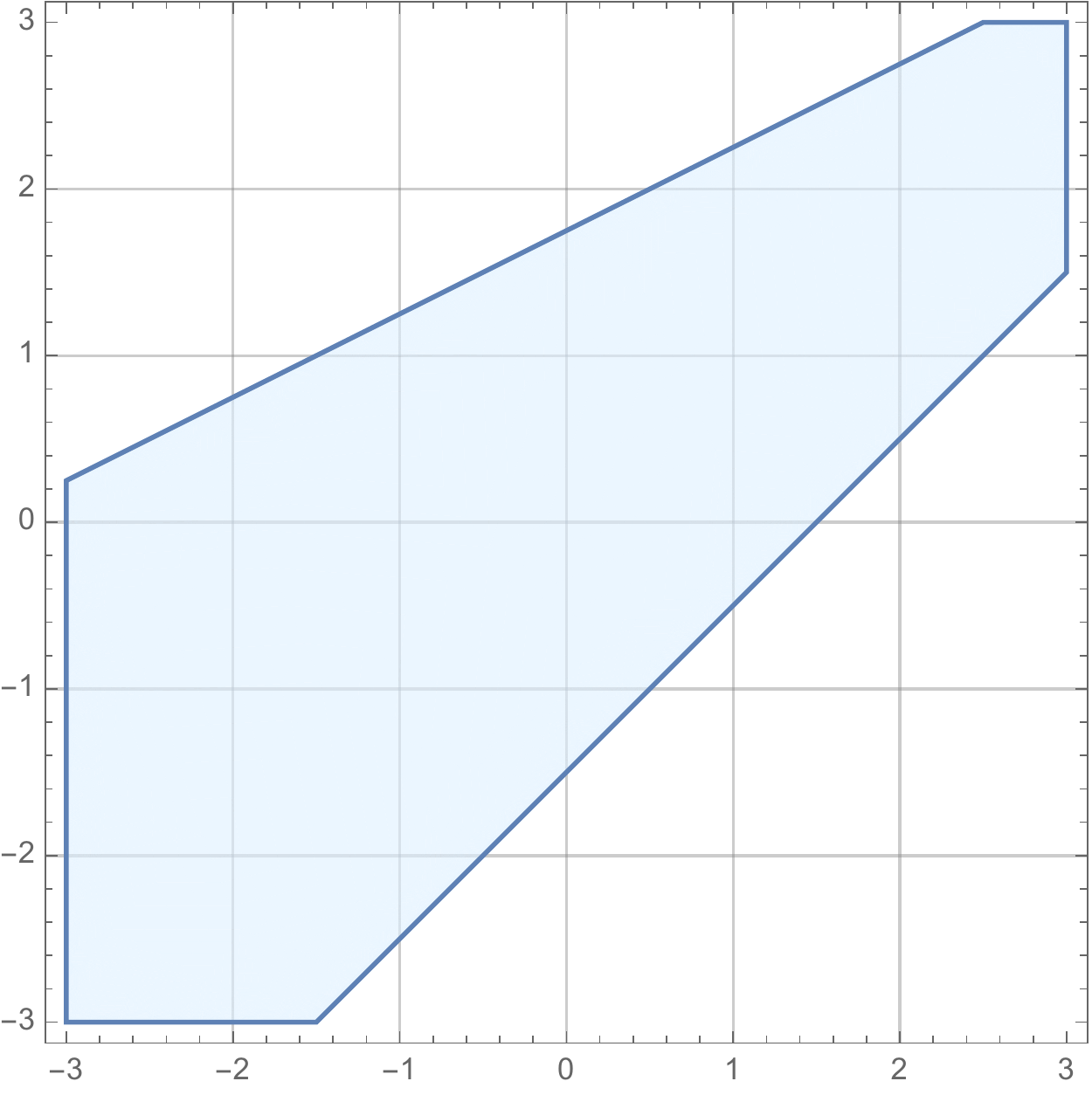}
  \end{minipage}
  \begin{minipage}{0.3\textwidth}
    \centering
    \includegraphics[width=0.7\textwidth]{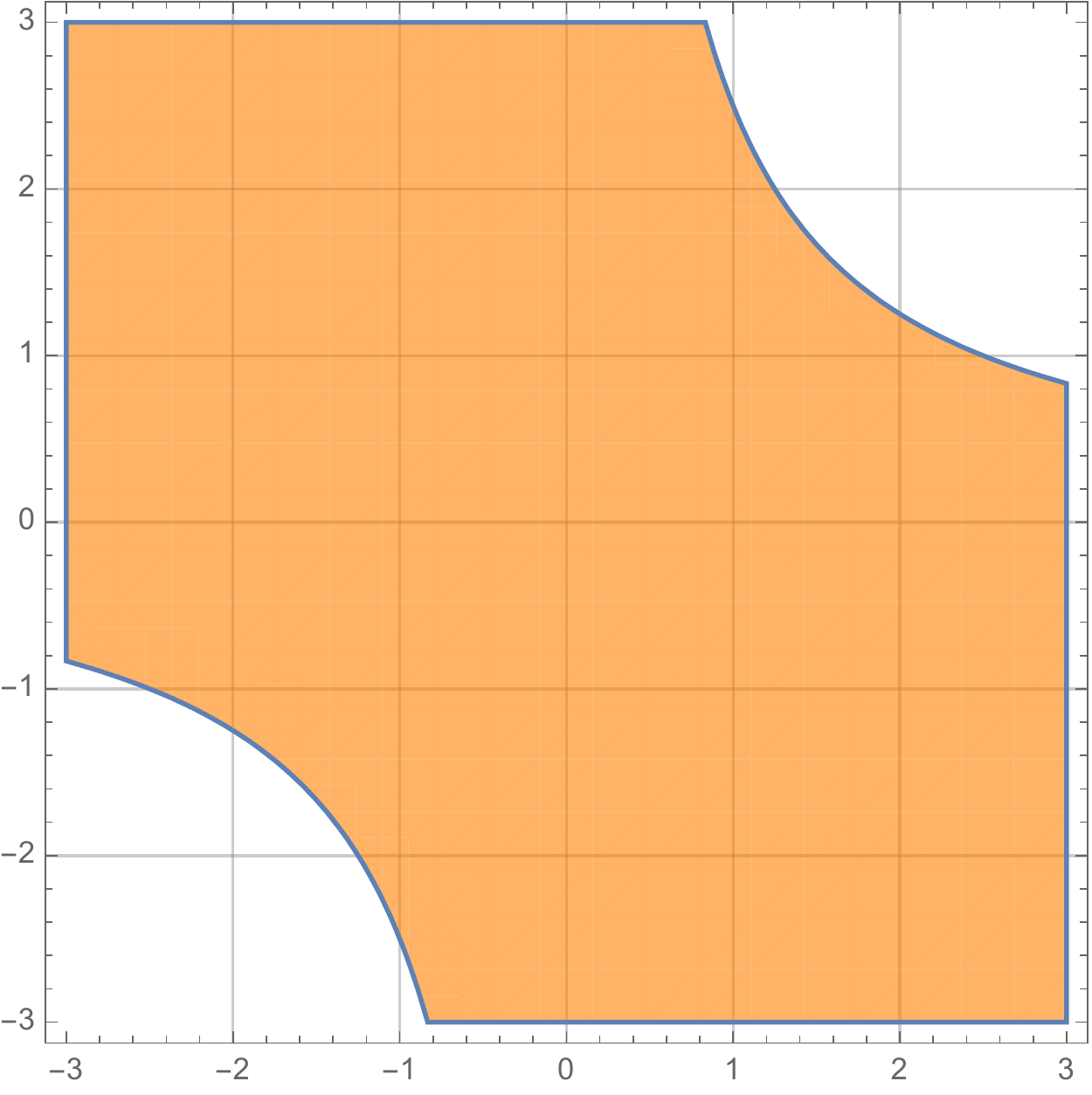}
  \end{minipage}
  \begin{minipage}{0.3\textwidth}
    \centering
    \includegraphics[width=0.7\textwidth]{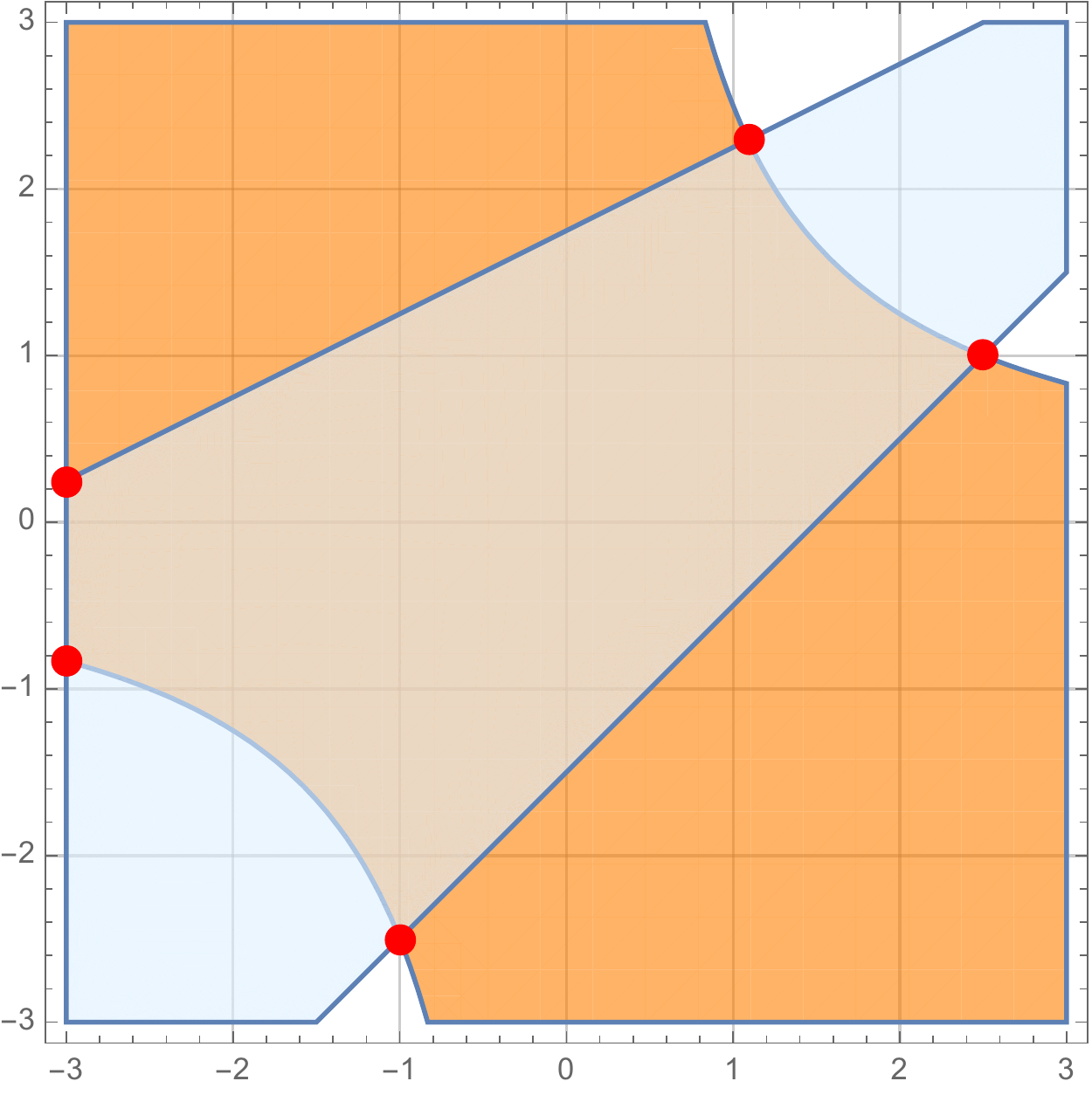}
  \end{minipage}
  \caption{ An example that shows bound reductions on $\xi$ and $\xj$ by utilizing $\Pij$ and bounds
    $[\lb_{ij},\ub_{ij}]$ on $\Xij$. The left plot shows $\Pij$, the middle plot shows the points $(\xi,\xj)$ that
    satisfy $\xi \xj \le \ub_{ij}$, and the right plot the intersection of both sets. Optimizing in the unit directions
    of the intersection, i.e., minimizing and maximizing $\xi$ and $\xj$, is equivalent to optimizing over the red
    points. }
  \label{fig:reverseprop}
\end{figure}

\section{Computational experiments}
\label{section:experiments}

In this section, we present a computational study of the presented propagation and separation ideas for bilinear terms for publicly
available instances of the \minlplib~\cite{MINLPLIB}.
We conduct three experiments to answer the following questions:
\begin{enumerate}
  \item \expAffected{}: Since it is unclear whether and to what extend \MINLPs in practice allow for a nontrivial projection $\projrelaxset$,
  we first investigate empirically how many instances have a linear relaxation that provides inequalities of the form~\eqref{eq:projineq}
  that are not axis-parallel.
  \item \expRootgap{}: How much gap can be closed when using the stronger separation and propagation of bilinear terms
  only in the root node of a branch-and-bound tree with aggressive root separation settings?
  \item \expTree{}: How much do the presented techniques affect the solvability and performance of \MINLPs in spatial branch-and-bound? For
    this experiment, we discuss suitable working limits on the number of \LP iterations to solve the projections and investigate the
    performance impact of the stronger separation and propagation individually.
\end{enumerate}

Our ideas are embedded in the \MINLP solver \scip~\cite{SCIP}. We refer
to~\cite{Achterberg2007a,Vigerske2013,Vigerske2017} for an overview of the general solving algorithm and \MINLP features
of \scip.

\subsection{Experimental setup}

For the \expAffected{} and \expRootgap{} experiments, we disable the \LP~iteration limit of the \OBBT propagator, enable the aggressive
separation emphasis setting, and disable restarts.\footnote{
\revision{In a restart, \scip aborts the current search process and preprocesses the problem again. Per default, this only happens in
the root node when enough variable bound reductions could be found. We refer to~\cite[Section 10.9]{Achterberg2007a} for more details
about restarts.}}\footnote{\scip
  settings \texttt{propagating/obbt/itlimfactor = -1}, \texttt{limits/restart =
    0}, \texttt{limits/totalnodes = 1}, and
  \texttt{separation/emphasis/aggressive = TRUE}}
The choices for the parameters ensure that the root node has been completely processed and there are no
further reductions possible by applying \OBBT again. Afterward, we use the current linear relaxation to compute the
two-dimensional projections $\Pij$ as described in Algorithm~\ref{algo:projections}. The projections are then used to
strengthen the separation and propagation of constraints of the form $\Xij = \xi \xj$.

In contrast to the first two experiments, the \expTree{} experiment is based on default settings.
The projections are utilized at every node of the
branch-and-bound tree. Note that the convex hull of the graph of $\xi \xj$ on $\Pij$ is in general not polyhedral. To
prevent a potential slowdown caused by too many separation rounds, at local nodes of the branch-and-bound tree, i.e., not at the root node, we use the inequalities only twice for separation. Additionally, we use a limit on the total number of
\LP iterations in order to bound to computational cost of solving~\eqref{eq:diaglp}. Similarly to Gleixner et
al.~\cite{Gleixner2017}, a limit of three times the \LP iterations that are spent so far at the root node is imposed.

For the \expAffected{} and \expRootgap{} experiments, we use a time limit of 7200s and a memory limit of 30~GB to ensure that for each instance
the root node could be completely processed. For our \expTree{} experiment, all instances run with a time limit of 1800s, a
memory limit of 30~GB, and an optimality gap limit of $10^{-4}$ to reduce the impact of tailing-off effects.

\paragraph{Implementation}

We extended two existing plug-ins of \scip: the \OBBT propagator, which can now additionally compute the two-dimensional
projections for variables that appear in a bilinear term $\xi\xj$; and a so-called nonlinear handler
 that calls Locatelli's algorithm and the propagation techniques described in
Section~\ref{section:propagation} for each $\xi\xj$ individually. Bilinear terms that only appear in convex constraints
or contain binary variables are ignored in both steps.
To reduce side effects, we use a separate working limit for solving the \LPs~\eqref{eq:diaglp} after applying standard
\OBBT. \revision{This is similar to the structure of Algorithm~\ref{algo:projections}.}

\revision{Using \OBBT in a local node of the tree search results in a significant slowdown of \scip. For this reason, by default,
\scip applies \OBBT only in the root node of the branch-and-bound tree.}

\paragraph{Test set}

We used the publicly available instances of the \minlplib~\cite{MINLPLIB}, which at time of the experiments
contained \ninstances{} instances.
This includes among others instances from the first \minlplib, the nonlinear programming library \globallib, and the
CMU-IBM initiative \href{www.minlp.org}{minlp.org}~\cite{MINLPDOTORG}.
We selected the instances that were available in OSiL format and consisted of nonlinear expressions that could be
handled by \scip: 1671~instances.

\paragraph{Hardware and software}

The experiments were performed on a cluster of 64bit Intel Xeon X5672~CPUs at 3.2\,GHz with 12\,MB cache and 48\,GB main
memory.
In order to safeguard against a potential mutual slowdown of parallel processes, we ran only one job per node at a time.
We used a development version of \scip that is based on version~6.0 with \cplex~12.8.0.0 as LP~solver~\cite{Cplex},
\cppad~20180000.0~\cite{CppAD}, and \ipopt~3.12.11 as \NLP solver~\cite{WachterBiegler2005,Ipopt} with
\mumps~4.10.0~\cite{Amestoy2001}.

\paragraph{Averages and statistical tests}

In order to evaluate algorithmic performance over a large set of benchmark instances, we compare geometric means, which
provide a measure for relative differences.
This avoids results being dominated by outliers with large absolute values as is the case for the arithmetic mean.
In order to also avoid an over-representation of differences among very small values, we use the shifted geometric mean.
The \emph{shifted geometric mean} of values $v_1,\ldots,v_N \geq 0$ with shift~$s \geq 0$ is defined as
\begin{equation*}
  \left(\prod_{i=1}^N (v_i + s)\right)^{1/N} - s.
\end{equation*}
See also the discussion in~\cite{Achterberg2007a,AchterbergWunderling2013,Hendel2014}. We use a shift value of $100$ for
\LP iterations and a value of one second for the solving time.

\subsection{Computational results}

In the following, we present results for the three above described experiments.

\paragraph{\expAffected{} experiment.}

In order to quantify how many instances are potentially affected by our ideas, we use the number of bilinear terms for which
a useful two-dimensional projection could be found after processing the root node. We prioritize bilinear terms that appear in
multiple quadratic constraints. In our analysis this is captured by taking the occurrence of a bilinear term in the
original \MIQCP~\eqref{eq:miqcp} into account. Denote by
\begin{equation*}
  K_{ij} := |\{k \in \consindex \mid (Q_k)_{ij} \neq 0\}|
\end{equation*}
the number of constraints in~\eqref{eq:miqcp} that contain $\xi\xj$. The value $\phi_{ij} \in \{0,1\}$ indicates whether
a useful projection could be found for $\xi\xj$ or not. Then
\begin{equation*}
  \Psi := \frac{\sum_{i,j} K_{ij} \phi_{ij}}{\sum_{i,j} K_{ij}} \in [0,1]
\end{equation*}
defines a measure for the \emph{effectiveness} of an \MIQCP. The interpretation of $K_{ij}$ in the definition of $\Psi$ is that
each bilinear term $\xi\xj$ is counted as a separate term of~\eqref{eq:miqcp}.

\begin{figure}[t]
  \centering
  \includegraphics[width=0.6\textwidth]{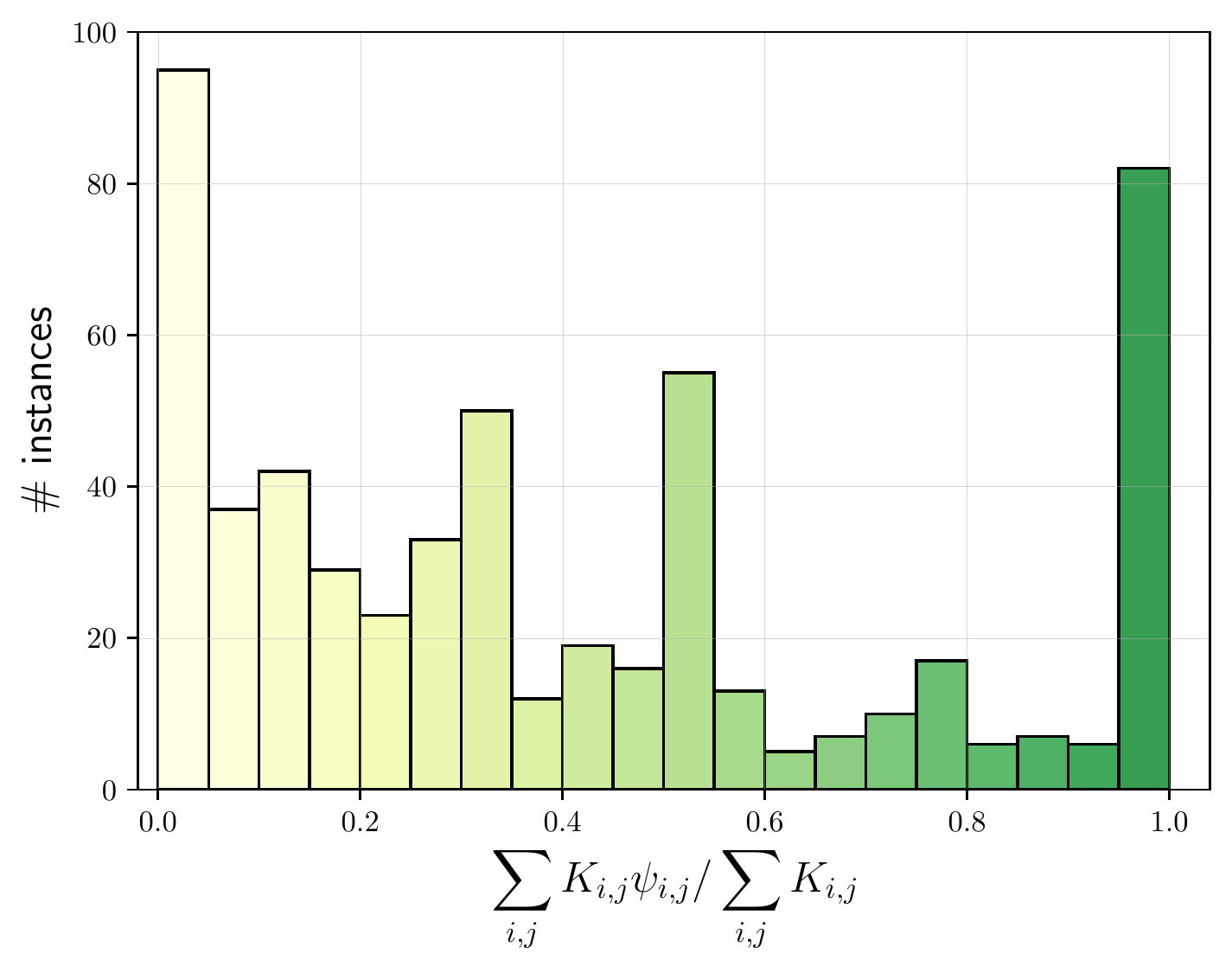}
  \caption{Instances of the \minlplib that contain at least one bilinear term $\xi\xj$ for which the two-dimensional
    projection is not equal to $\ijbox$. The y-axis displays the total number of instances and the x-axis an interval
    effectiveness $\Psi \in [0,1]$.}
  \label{fig:affected}
\end{figure}

Figure~\ref{fig:affected} shows the effectiveness on the instances of the \minlplib, where instances with $\Psi = 0$
are filtered out. Detailed results for all instances that contain at least one bilinear term are given in
Table~\ref{table:affected:detailed} of the electronic supplement. Out of the \ninstances{}, \affectedNNoBilins{} do not contain a bilinear term or
are solved before computing the two-dimensional projections. In total, \affectedNAffected{} instances provide a relevant
projection for at least one bilinear term, i.e., $\Psi > 0$. There are \affectedNAffectedZero{} instances with an
effectiveness between $0-5$\% and \affectedNAffectedFull{} instances with an effectiveness of $95-100$\%. The average
effectiveness among all instances is \affectedArithAll{} and \affectedArithAffected{} for the subset of instances that
have a strictly positive effectiveness.

Note that although we do not use an exact algorithm for computing the projection, we obtain the same number of relevant
instances because if no nontrivial facet was found then the box is the exact projection, i.e, $\ijbox = \projrelaxset$.

To analyze the computational cost of computing all projections, we use the total number of \LP iterations and the time
spent for solving all \LPs~\eqref{eq:diaglp}. Computing all projections takes on average \affectedGeomeanTimeAll{}
seconds and \affectedGeomeanItersAll{} \LP iterations. On instances that do not provide any useful projection, we
observe on average \affectedGeomeanItersNotffected{} \LP iterations and spend \affectedGeomeanTimeNotffected{} seconds
in computing the projections. This time can be considered to be a constant slow-down because we could not learn anything
for these instances which could pay off in the remaining solution process. For instances with a strict positive
effectiveness, we use on average \affectedGeomeanItersAffected{} \LP iterations and \affectedGeomeanTimeAffected{}
seconds.

We briefly report on the success of filtering candidates by exploiting previously computed \LP optima \revision{in
Line~\ref{algo:projections:filter1} and~\ref{algo:projections:filter2} of Algorithm~\ref{algo:projections}}. Out of all
\ninstances{} instances, we could filter candidates on \affectedFilteredInstances{} instances. On these instances, the
filtering rate is on average \affectedFilteredRateAllPercentage{}\% and \affectedFilteredRateFilteredPercentage{}\% on
the \affectedNAffected{} selected instances.

\revision{Last, we report on the impact of finding nontrivial inequalities when applying Algorithm~\ref{algo:projections}
  multiple times in the root node. As discussed in Section~\ref{section:projections:computing}, tighter projections could be
  found when refining $\relaxset$ after calling \OBBT. Indeed, we observed a slight improvement in the success when
  recomputing the projections. The first bar of Figure~\ref{fig:affected} decreases from $\affectedNAffectedZero{}$ to $87$, which
  means that for $10$ more instances a relevant projection could be found that could not be found before. The
  average effectiveness improves from $\affectedArithAllPercent$\% to $19.2$\% on all instances, and improves from
  $\affectedArithAffectedPercent$\% to $41.0$\% on the affected instances.}

\revision{Even though there is a slight improvement in the success when recomputing the projections in the root node, we observed
  that the tighter projections have almost no impact on the dual bounds of the \expRootgap{} experiment. Due to the fact that
  recomputing the projections can be expensive, we only use Algorithm~\ref{algo:projections} once in the root node.}

\paragraph{\expRootgap{} experiment}

Aggregated results for the \expRootgap{} experiment are shown in Table~\ref{table:gapclosed} and visualized in
Figure~\ref{fig:gapclosed}. We refer to Table~\ref{table:root:detailed} in the electronic supplement for detailed, instance-wise results.

From the potentially \affectedNAffected{} affected instances of the previous experiment, we filtered out all instances
that have been detected to be infeasible, no primal solution is known, or we could not prove any finite dual bound with
the above described settings. This leaves \rootNinstances{} instances. Let $I := \{1,\ldots,\rootNinstances{}\}$ be the
index set of these instances.

\begin{definition}\label{definition:gapclosed}\normalfont
  Let $p \in \R$ be a valid primal bound and $d_1 \in \R$ and $d_2 \in \R$ be two dual bounds for~\eqref{eq:miqcp},
  i.e., $d_1 \le p$ and $d_2 \le p$. The function $\gapfct : \R^3 \rightarrow [-1,1]$ with
  \begin{equation*}
    \gapfct(p,d_1,d_2) :=
    \begin{cases}
      0, & \text{ if } d_1 = d_2 \\
      +1 - \frac{p - d_1}{p - d_2}, & \text{ if } d_1 > d_2 \\
      -1 + \frac{p - d_2}{p - d_1}, & \text{ if } d_1 < d_2
    \end{cases}
  \end{equation*}
  measures the \emph{gap closed} improvement when comparing the distance of $d_1$ and $d_2$ to $p$.
\end{definition}

Denote by $d_1^i$ and $d_2^i$ the dual bounds of instance $i \in I$ obtained with and without using the two-dimensional
projections for separation and propagation. A reference primal bound $p^i$ is given by the best known bound for
$i \in I$ in the \minlplib. We use the gap-closed values for comparing the bounds $d_1^i$ and $d_2^i$ with respect to
$p^i$. Note that $\gapfct(d_1^i,d_2^i,p^i) = 1$ implies $d_1^i = p^i$ and $d_2^i < d_1^i$, which means that the instance
could be solved in the root node to optimality when using the two-dimensional
projections, \revision{but could not be solved to optimality in the root node without them.}

Table~\ref{table:gapclosed} shows that using the projections for separation and propagation has a significant impact on
the quality of the achieved dual bounds in the root node. On all \rootNinstances{} instances, the average gap closed
improvement is \rootGapClosedAllPercent{}\%. The average improvement is \rootGapClosedAffectedPercent{}\% on
\rootGapClosedAffectedNInstances{} instances for which the gap closed values differ by at least 1\%. Considering the
affected instances with a minimum improvement or deterioration of 1\% reveals that the dual bounds
improve on \rootGapClosedBetterNInstances{} and only get worse on \rootGapClosedWorseNInstances{} instances. The average gap
improvement is \rootGapClosedBetterPercent{}\% on the \rootGapClosedBetterNInstances{}~instances and
\rootGapClosedWorsePercent{}\% on the \rootGapClosedWorseNInstances{} instances.

\revision{Next, we briefly report on the three instances in Figure~\ref{fig:gapclosed} that have a gap closed value less than~$-80$\%.
  Those instances are \texttt{crudeoil\_lee4\_05}, \texttt{crudeoil\_lee4\_06}, and \texttt{nuclear25b}. The dual bounds obtained
  for both \texttt{crudeoil} instances are $d_1 = 132.585$ and $d_2 = 132.548$,
  and the dual bounds for \texttt{nuclear25b} are $d_1 = -1.74673$ and $d_2 = -1.2208$.
  The primal bounds are $132.548$ for both \texttt{crudeoil} instances and $-1.1136$ for \texttt{nuclear25b}
  None of the three
  instances run into the time limit, which means that the differences in the dual bounds are caused by side effects or internal working limits
  in \scip. Interestingly, it can be observed that \scip applies $1.5$ to $2$ times more cutting planes when deactivating our developed methods for
  those three instances. However, we could not observe that the performance degradation is causally related to the new methods.}

\begin{figure}[t]
  \centering
  \includegraphics[width=0.6\textwidth]{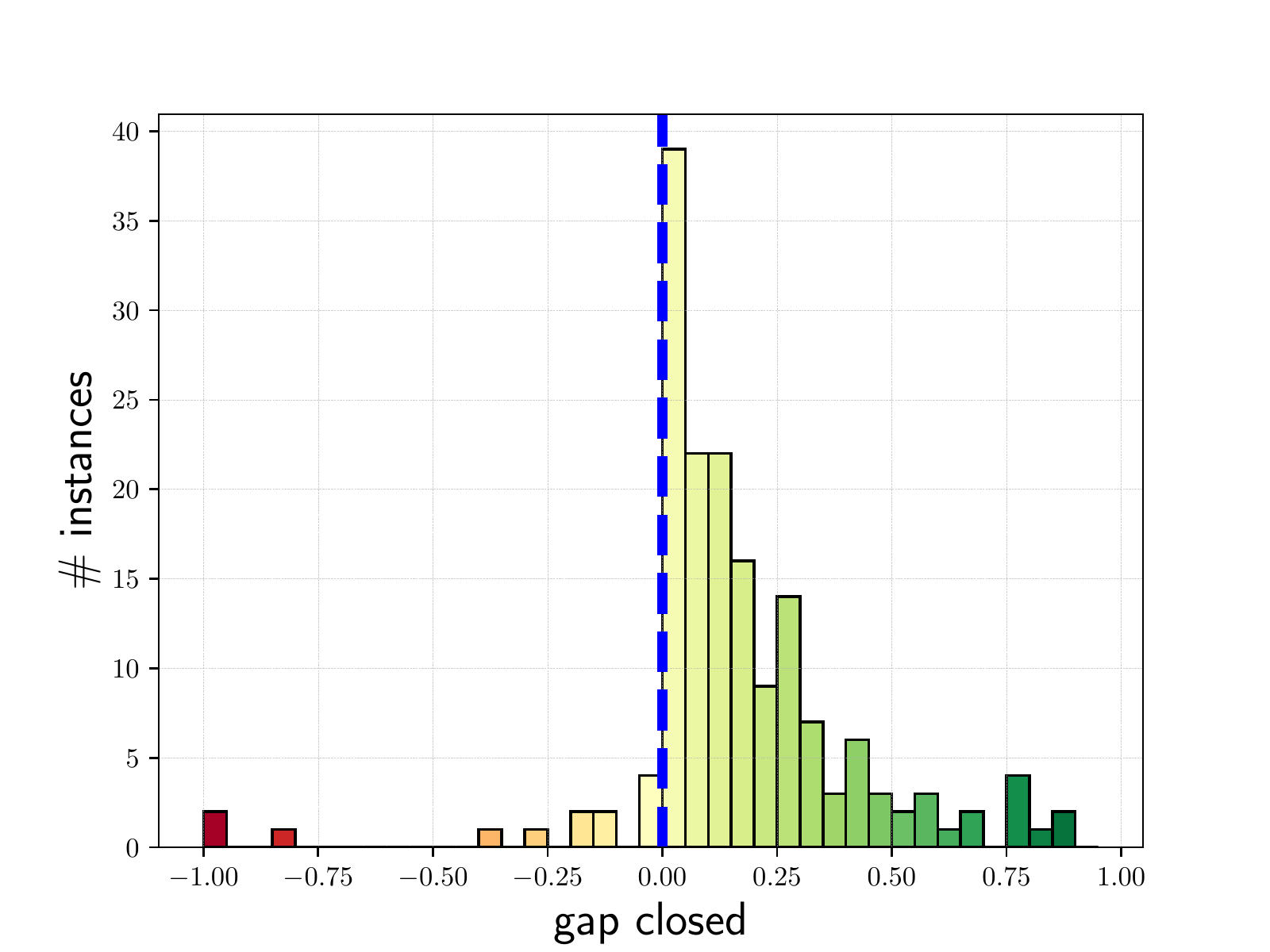}
  \caption{A bar diagram that visualizes the gap closed improvements for the \rootNinstances{} selected instances. Each
    bar maps to an interval with width 0.05 that corresponds to the gap closed improvement. The height of the bar
    displays the total number of instances that achieved a value in the corresponding interval.}
  \label{fig:gapclosed}
\end{figure}

\begin{table}[t]
  \centering
  \begin{tabular*}{0.5\textwidth}{@{\extracolsep{\fill}}lrr}
    \toprule
    & \# instances & gap closed \\
    \midrule
    \texttt{ALL}         & \rootNinstances{}                  & \rootGapClosedAllPercent{}\% \\ \midrule
    \texttt{>1\% change} & \rootGapClosedAffectedNInstances{} & \rootGapClosedAffectedPercent{}\% \\
    \texttt{>1\% better} & \rootGapClosedBetterNInstances{}   & \rootGapClosedBetterPercent{}\% \\
    \texttt{>1\% worse} &  \rootGapClosedWorseNInstances{}    & \rootGapClosedWorsePercent{}\% \\
    \bottomrule
  \end{tabular*}
  \caption{Aggregated results for the \expRootgap{} experiments. \revision{The table shows the average gap closed values for
    different subsets of instances.} \texttt{ALL} contains all instances, \texttt{>1\% better} instances that improved by
    at least one percent, \texttt{>1\% worse} instances that deteriorate by at least one percent, and \texttt{>1\%
      change} is the union of instances in \texttt{>1\% better} and \texttt{>1\% worse}.}
  \label{table:gapclosed}
\end{table}

\paragraph{\expTree{} experiment}

For the \expTree{} experiment, we use five permutations for each of the $\treeNInstancesAll{}$ instances per setting in order to robustify
the results against performance variability~\cite{KochEtAl2011,Lodi2013}. \revision{A permutation of an instance randomly changes the order of the variables and the
constraints. This can have a large impact on the behavior and the performance of a \MINLP solver.} An instance is considered to be solved
by a setting if all permutations could be solved by this setting. Hence, if a setting solves more instances it means that it could consistently
solve more instances over all permutations. For comparing solving times between different settings, we use the shifted geometric mean with a
shift value of one second for the five permutations of an instance and then consider the shifted geometric mean of all these values.

Aggregated results for the tree experiments are shown in Table~\ref{table:treeresults} and more detailed results for
each instance are contained in Table~\ref{table:tree:detailed} of the electronic supplement. \scip with its default settings solves
\treeNSolvedAllOff{} of the \treeNInstancesAll{} instances. When activating the use of projections for separation $3$ more
instances are solved than with default \scip; when activating it for both separation and propagation $5$ more instances are solved. Considering the total time, we
see that on average \treeScipSepa{} and \treeScipSepaProp{} is $3\%$ faster than \treeScip{}. The groups
\texttt{[1,tlim]}, \texttt{[10,tlim]}, and \texttt{[100,tlim]} are the subsets of instances for which at least one
setting solved the instance in \revision{more than} one, ten, or $100$ seconds, respectively.
These subsets form a hierarchy of increasingly difficult instance sets in an unbiased manner.
Compared to \treeScip{}, \treeScipSepaProp{}
solves $4$ more instances on \texttt{[1,tlim]}, $4$ more on \texttt{[10,tlim]}, and $5$ more on
\texttt{[100,tlim]}. With respect to time, \treeScipSepaProp{} is $11\%$ faster on \texttt{[1,tlim]}, $18\%$ on
\texttt{[10,tlim]}, and even $36\%$ on \texttt{[100,tlim]} than \treeScip{}.

A comparison of the second and the third column of Table~\ref{table:treeresults} shows that both the separation and the
propagation contribute to the larger number of solved instances. While activating separation alone does not improve the solving time, it can be seen that, more importantly, it does help to solve
more instances in total.

\begin{table}[t]
  \centering
  \begin{tabular*}{0.9\textwidth}{@{\extracolsep{\fill}}lrrrrrr}
    \toprule
    & & \multicolumn{1}{c}{\treeScipSepaProp} & \multicolumn{2}{c}{\treeScipSepa} & \multicolumn{2}{c}{\treeScip} \\
    \cmidrule[0.5pt]{3-3} \cmidrule[0.5pt]{4-5} \cmidrule[0.5pt]{6-7}
    & n & \# solved & \# solved & time & \# solved & time \\
    \midrule
    \texttt{ALL}        & \treeNInstancesAll{}     & \treeNSolvedAllDefault{}     & \treeNSolvedAllNoprop{}     & $\treeTimeAllNoprop{}$     & \treeNSolvedAllOff{}     & $\treeTimeAllOff{}$ \\ \midrule
    \texttt{[1,tlim]}   & \treeNInstancesOne{}     & \treeNSolvedOneDefault{}     & \treeNSolvedOneNoprop{}     & $\treeTimeOneNoprop{}$     & \treeNSolvedOneOff{}     & $\treeTimeOneOff{}$ \\
    \texttt{[10,tlim]}  & \treeNInstancesTen{}     & \treeNSolvedTenDefault{}     & \treeNSolvedTenNoprop{}     & $\treeTimeTenNoprop{}$     & \treeNSolvedTenOff{}     & $\treeTimeTenOff{}$ \\
    \texttt{[100,tlim]} & \treeNInstancesHundred{} & \treeNSolvedHundredDefault{} & \treeNSolvedHundredNoprop{} & $\treeTimeHundredNoprop{}$ & \treeNSolvedHundredOff{} & $\treeTimeHundredOff{}$ \\
    \bottomrule
  \end{tabular*}
  \caption{Aggregated results of \scip on the \affectedNAffected{} potentially affected instances of the
    \minlplib. For each instance, five permutation including the default permutation have been solved. An instance is
    considered to be solved when all permutations of this could be solved by a setting. Three different settings for
    \scip are used: default settings \treeScip{}, \treeScipSepa{} for activating separation, and \treeScipSepaProp{} for
    activating separation and propagation for bilinear terms that provide a useful two-dimensional projection. \revision{The column
    ``time'' reports the change of solving time relative to \treeScipSepaProp{}.}}
  \label{table:treeresults}
\end{table}

\section{Conclusion}
\label{section:conclusion}

In this article, we presented techniques to improve the separation and propagation of bilinear terms when solving \MINLPs
with spatial branch-and-bound and gave an extensive computational study on a large heterogeneous test set. Our ideas
are based on projecting a linear relaxation onto two variables that appear bilinearly by solving a sequence of \LPs that
are similar to the ones in \OBBT. Instead of computing the full projection, we compute a relaxation of the projection
that is described by few inequalities. By applying known polyhedral results, we are able to strengthen the separation of
quadratic constraints by computing the convex and concave envelope of $\xi\xj$ on the two-dimensional
projections. Additionally, we presented that the projections also enables us to tighten variable bounds. Computing the best
possible bounds of $\xi$, $\xj$, and $\xi\xj$ on the projection is in general a nonconvex optimization problem. We
proved that these problems can be efficiently solved by computing a discrete set of points. This allows us to
efficiently solve these optimization problems at every node of the branch-and-bound tree.

Our experiments on the publicly available instances of the \minlplib based on an implementation in the \MINLP solver
\scip show that \affectedNAffected{} of the \ninstances{} instances provide nontrivial projections for at least
one bilinear term. On these instances, it was possible to compute useful projections for
\affectedArithAffectedPercent{}\% of all bilinear terms. When using the projection exhaustively during the separation of the
root node, we observed an improvement of the achieved dual bounds on \rootGapClosedBetterNInstances{} and a deterioration on
only \rootGapClosedWorseNInstances{} instances. The average gap closed improvement on all instances for which a change
of at least one percent could be observed is \rootGapClosedAffectedPercent{}\%.
Finally, our tree experiments showed that the new techniques improve performance by $36$\% on difficult instances
and enable us to consistently solve more instances.

There are two interesting extensions of the presented methods. First, our propagation techniques do not only apply to
polyhedral projections but also for general two-dimensional convex sets. \revision{How to compute} these convex sets efficiently by
using a convex relaxation of a \MINLP remains an open question.
Second, for models that contain symmetric structures the tightness of the
two-dimensional projections and the performance improvements gained might profit
particularly from symmetry-breaking constraints of the form $\xi \le \xj$. These
inequalities are in general not implied by a linear relaxation, but can be
derived by considering formulation symmetry~\cite{Leo2011}.
 
\section*{Acknowledgments}
This work has received support from the Federal Ministry of Education and
Research (BMBF Grant~05M14ZAM, Research Campus MODAL) and from the Federal
Ministry for Economic Affairs and Energy (BMWi grant 03ET1549D, project EnBA-M).
All responsibilty for the content of this publication is assumed by the authors.
The authors thank the Schloss Dagstuhl – Leibniz Center for Informatics for hosting
the Seminar 18081 "Designing and Implementing Algorithms for Mixed-Integer Nonlinear Optimization"
for providing the environment to develop the ideas in this paper.
 
\bibliographystyle{spmpsci}
\bibliography{bilin}

\newpage
\listoftodos[Notes]

\begin{longtable}{@{\extracolsep{\fill}}lrrrrrr}
  \caption{Detailed results for the effectiveness of instances of the \minlplibtwo that contain bilinear terms. All
    instances for which Algorithm~\ref{algo:projections} has not been called are filtered out.}
  \label{table:affected:detailed}\\
  \tabledefline{\# \LPs}{total number of solved \LPs}
  \tabledefline{\# iters}{total number of \LP iterations used}
  \tabledefline{\# filtered}{total number of filtered candidates}
  \tabledefline{time}{time spent for solving all \LPs (in seconds)}
  \toprule
  Instance & $\sum_{i,j} K_{ij} \phi_{ij}$ & $\sum_{i,j} K_{ij}$ & \# \LPs & \# iters & \# filtered & time \\
  \midrule
  \endfirsthead
  \textbf{Table~\ref{table:affected:detailed}} continued\\
  \toprule
  Instance & $\sum_{i,j} K_{ij} \phi_{ij}$ & $\sum_{i,j} K_{ij}$ & \# \LPs & \# iters & \# filtered & time \\
  \midrule
  \endhead
  \bottomrule
  \endfoot
                     \texttt{spar30-100-1} &     0 &   427 &   376 & 29724 &   878 &  0.69 \\
                   \texttt{spar30-100-2} &     0 &   430 &   493 & 43747 &   692 &  1.62 \\
                   \texttt{spar30-100-3} &     0 &   428 &   375 & 30584 &   912 &  0.71 \\
                    \texttt{spar30-60-1} &     0 &   250 &   178 & 10125 &   608 &  0.19 \\
                    \texttt{spar30-60-2} &    11 &   240 &   313 & 16239 &   260 &  0.60 \\
                    \texttt{spar30-60-3} &     0 &   288 &   410 & 22576 &   298 &  0.47 \\
                    \texttt{spar30-70-1} &     0 &   290 &   246 & 16683 &   600 &  0.51 \\
                    \texttt{spar30-70-2} &     0 &   288 &   376 & 22727 &   351 &  0.69 \\
\end{longtable}

\begin{longtable}{@{\extracolsep{\fill}}lrrrrr}
  \caption{Detailed results for the root gap experiments. Dual bounds are highlighted if there is an improvement or
    deterioration of at least one percent. Aggregated results are presented in Table~\ref{table:gapclosed}.}
  \label{table:root:detailed}\\
  \tabledefline{$\gapfct(d_1,d_2,p)$}{gap closed improvement}
  \tabledefline{$d_1$}{dual bound obtained with additional separation and propagation for bilinear terms}
  \tabledefline{$d_2$}{dual bound obtained without additional separation and propagation for bilinear terms}
  \tabledefline{$p$}{reference primal bound}
  \tabledefline{obj}{objective sense}
  \toprule
  Instance & $\gapfct(d_1,d_2,p)$ & $d_1$ & $d_2$ & $p$ & obj \\
  \midrule
  \endfirsthead
  \textbf{Table~\ref{table:root:detailed}} continued\\
  \toprule
  Instance & $\gapfct(d_1,d_2,p)$ & $d_1$ & $d_2$ & $p$ & obj \\
  \midrule
  \endhead
  \bottomrule
  \endfoot
      \texttt{10bar1C} &  0.00 &    1485.26 &    1485.26 &    1623.09 &   min\\
     \texttt{10bar3} &  0.00 &    1705.74 &    1705.74 &    5156.64 &   min\\
     \texttt{10bar4} &  0.00 &    1934.78 &    1934.78 &    5647.06 &   min\\
    \texttt{4stufen} &  0.00 &     100948 &     100948 &     116330 &   min\\
      \texttt{90bar} &  0.00 &    90.3833 &    90.3619 &    97.5374 &   min\\
 \texttt{alkylation} &  0.00 &    2374.08 &    2374.08 &    1768.81 &   max\\
      \texttt{alkyl} & \textbf{ 0.01} & \textbf{  -1.99632} &   -1.99912 &     -1.765 &   min\\
   \texttt{arki0005} &  0.00 &          0 &          0 &    372.605 &   min\\
\end{longtable}

\begin{longtable}{@{\extracolsep{\fill}}lrrrrrr}
  \caption{Detailed results for the tree experiments. The table reports shifted geometric means for the number of
    branch-and-bound nodes and the solving time for five different permutations for each of the~564 relevant instances
    of the affected instances experiment, see Table~\ref{table:affected:detailed}. Aggregated results are presented in
    Table~\ref{table:treeresults}.}
  \label{table:tree:detailed}\\
  \tabledefline{\treeScipSepaProp}{\scip using stronger separation and propagation for bilinear terms}
  \tabledefline{\treeScipSepa}{\scip using stronger separation for bilinear terms}
  \tabledefline{\treeScip}{default \scip}
  \tabledefline{$t$}{shifted geometric mean of solving times (shift value $1.0$)}
  \tabledefline{nodes}{shifted geometric mean of branch-and-bound nodes (shift value $100.0$)}
  \toprule
  & \multicolumn{2}{c}{\treeScipSepaProp} & \multicolumn{2}{c}{\treeScipSepa} & \multicolumn{2}{c}{\treeScip} \\ \cmidrule(){2-3} \cmidrule(){4-5} \cmidrule(){6-7}
  Instance & $t$ & nodes & $t$ & nodes & $t$ & nodes\\
  \midrule
  \endfirsthead
  \textbf{Table~\ref{table:tree:detailed}} continued\\
  \toprule
  & \multicolumn{2}{c}{\treeScipSepaProp} & \multicolumn{2}{c}{\treeScipSepa} & \multicolumn{2}{c}{\treeScip} \\ \cmidrule(){2-3} \cmidrule(){4-5} \cmidrule(){6-7}
  Instance & $t$ & nodes & $t$ & nodes & $t$ & nodes\\
  \midrule
  \endhead
  \bottomrule
  \endfoot
                         10bar1C &        1.2 &      246.8 &        1.2 &      246.8 &        1.2 &      246.8 \\
                        10bar3 &      142.5 &       18.9K &       85.3 &       81.1K & \textbf{      79.6} &       73.1K \\
                        10bar4 &      296.0 &      172.4K &      274.6 &      170.6K & \textbf{     239.9} &      149.7K \\
                       4stufen &     1800.0 &      215.9K &     1800.0 &      210.8K &     1800.0 &      224.5K \\
                         90bar & \textbf{     640.3} &       23.6K &      723.3 &       26.4K &      820.0 &       33.6K \\
                         alkyl &        0.8 &       99.0 &        0.8 &       95.0 &        0.7 &       95.0 \\
                    alkylation &     1800.0 &     2486.9K &     1800.0 &     2499.5K &     1800.0 &     2505.3K \\
                      arki0005 &     1800.0 &        1.0 &     1800.0 &        1.0 &     1800.0 &        1.0 \\
\end{longtable}
 
\end{document}